\documentclass[11pt]{article}

\usepackage[utf8]{inputenc} 
\usepackage[T1]{fontenc}    
\usepackage{hyperref}       
\usepackage{url}            
\usepackage{booktabs}       
\usepackage{amsfonts}       
\usepackage{nicefrac}       
\usepackage{microtype}      
\usepackage{xcolor}         
\usepackage[margin=.9in]{geometry}
\usepackage{bbm}  

\usepackage{amsmath}
\usepackage{amssymb}
\usepackage{mathtools}
\usepackage{amsthm}
\usepackage{stmaryrd}

\usepackage{algorithmic}
\usepackage[linesnumbered,ruled,vlined]{algorithm2e}

\definecolor{burgundy}{rgb}{0.5,0.0, 0.13}
\hypersetup{colorlinks,
citecolor=burgundy,
citebordercolor=burgundy,
linkcolor=burgundy,
urlcolor=MidnightBlue
}

\usepackage{caption}
\usepackage{subcaption}
\newtheorem{theorem}{Theorem}[section]
\newtheorem{proposition}[theorem]{Proposition}
\newtheorem{lemma}[theorem]{Lemma}
\newtheorem{corollary}[theorem]{Corollary}
\theoremstyle{definition}

\theoremstyle{remark}
\newtheorem{remark}[theorem]{Remark}

\newcommand{\jz}[1]{}
\newcommand{\gd}[1]{}
\newcommand{\ml}[1]{}

\newcommand{\cA}{\mathcal{A}}

\def\d{{\mathrm d}}
\def\F{{\mathcal F}}
\def\E{{\mathbb E}}
\newcommand{\cP}{\mathcal{P}}
\newcommand{\cW}{\mathcal{W}}
\newcommand{\cX}{\mathcal{X}}
\newcommand{\EE}{\mathbb{E}}
\newcommand{\RR}{\mathbb{R}}
\newcommand{\VV}{\mathbb{V}}

\title{Deep Learning for Population-Dependent Controls \\
            in Mean Field Control Problems with Common Noise}

\begin{document}
 \author{
G\"{o}k\c{c}e Dayan{\i}kl{\i} \thanks{Department of Statistics, University of Illinois at Urbana-Champaign, Champaign, IL 61820, USA. Email: gokced@illinois.edu}
\and Mathieu Lauri\`ere \thanks{Shanghai Frontiers Science Center of Artificial Intelligence and Deep Learning; NYU-ECNU Institute of Mathematical Sciences at NYU Shanghai; NYU Shanghai, 567 West Yangsi Road, Shanghai, 200126, People’s Republic of China. Email: mathieu.lauriere@nyu.edu.}
\and Jiacheng Zhang
 \thanks{Department of Industrial Engineering and Operations Research, University of California, Berkeley, Berkeley, California, USA. 
 Email:jiachengz@berkeley.edu }}
\date{}

\maketitle

\begin{abstract}
In this paper, we propose several approaches to learn the optimal population-dependent controls in order to solve mean field control problems (MFC). Such policies enable us to solve MFC problems with forms of common noises at a level of generality that was not covered by existing methods. We analyze rigorously the theoretical convergence of the proposed approximation algorithms. Of particular interest for its simplicity of implementation is the $N$-particle approximation. The effectiveness and the flexibility of our algorithms is supported by numerical experiments comparing several combinations of distribution approximation techniques and neural network architectures. We use three different benchmark problems from the literature: a systemic risk model, a price impact model, and a crowd motion model. We first show that our proposed algorithms converge to the correct solution in an explicitly solvable MFC problem. Then, we show that population-dependent controls outperform state-dependent controls. Along the way, we show that specific neural network architectures can improve the learning further.
\end{abstract}


\section{Introduction}

Optimal control problems have found a wide range applications from engineering to finance and robotics. In most cases, the system is subject to random disturbances which means that one has to find optimal controls in a stochastic setting. Several methods have been developed for such problems, such as Bellman's dynamic programming and Pontryagin's maximum principle. While stochastic optimal control is typically limited to one system or a small number of interacting systems (e.g., robots), the framework has recently been extended to the mean field setting. The main motivation of this setting is to study very large populations of strategic identical agents who cooperate to minimize a social cost. The mean-field approximation consists in replacing individual interactions by the interaction of a representative agent with the distribution of the population. This leads to more tractable models and more efficient algorithms. The setting is often referred to as mean 
field control (MFC for short)~\cite{bensoussan2013meanbook}. See \cite{carmona2018probabilistic,carmona2018probabilistic2} for a more comprehensive review.

Numerical methods to compute an optimal control generally rely on backward partial differential equations (PDEs) or backward stochastic differential equations (SDEs). In the MFC setting, these backward equations need to be coupled with forward equations in order to characterize the evolution of the population. The solutions to forward and backward differential equations have been numerically implemented using traditional methods such as finite differences and, more recently, using neural networks (NN) with for instance the Deep Backward SDE method~\cite{ehanjentzen2017deepbsde}, the deep Galerkin method~\cite{sirignano2018dgm} or physics-informed neural networks~\cite{raissi2019physics}. 
In the context of mean field games and control problems, PDEs have been solved using finite-difference schemes~\cite{achdoucapuzzo2010mean,briceno2018proximal} and deep learning methods~\cite{al2018solving,carmona2021convergenceergo}. Deep learning methods for McKean-Vlasov forward-backward SDE systems have also been proposed~\cite{fouque2020deep,carmona2022convergence,aurell2022optimal}. We refer to \cite{hu2023recent} for a recent review. However, the forward-backward structure leads to numerical challenges. Here, we focus on a simpler approach, in which directly aim for learning the optimal control without using backward PDEs or SDEs. This approach has been used previously in standard optimal control problems, e.g. by~\cite{gobet2005sensitivity,han2016deep-googlecitations}, and extended to the mean field setting in~\cite{fouque2020deep,carmona2022convergence}. One of the main advantages is the fact that it does not require any dynamic programming principle, which is known to be challenging to exploit for an MFC problem because it requires solving the problem for all possible distributions, which is not feasible. Learning directly the control through Monte Carlo simulation makes it possible to train the neural network on regions of the space that matter the most. 

When the agents are only subject to idiosyncratic randomness, this randomness vanishes in the mean-field limit and does not affect the evolution of the distribution, which is thus deterministic. In this setting, it is sufficient to learn controls that are functions of the representative agent's state and depend on the distribution only through the time step. 
However, when there is a common source of randomness affecting the whole population, this is no longer true. The evolution of the distribution cannot be predicted with certainty and, to be optimal, it becomes necessary to let the control be a function of the distribution. This question has thus far be little studied, with very few exceptions, such as~\cite{perrin2022generalization}, which introduced master policies in the context of mean field games, \cite{germain2022deepsets}, which focuses on a backward scheme to solve PDEs, and \cite{carmonalaurieretan2019model,gu2021meanfieldcontrolsQ} which develop reinforcement learning (RL) algorithms for mean field Markov decisions processes.

The main contribution of this article is three-fold. Firstly, we establish theoretical approximation guarantees for the population-dependent algorithms dealing with MFCs. In particular, we provide a bound on the sensitivity of the optimal cost when the mean-field distribution is replaced by an approximation in the cost, the dynamics and also the control function. This is about the stability of the problem with respect to (in principle general) distribution approximation. A direct application to numerical methods is when using an $N$-particle approximation for McKean-Vlasov dynamics, which is common in the literature. Secondly, building on this approximation theory, we present an algorithm which trains a neural network to minimize the social cost, and we propose several variants of distribution approximation (empirical, moments, histogram) and NN architectures (feedforward fully connected, convolutional, symmetric). Thirdly, we illustrate the performance of the various distribution approximations and architectures on three examples from the literature. We show that, in the presence of common noise, population-dependent controls outperform population-independent controls, and that the choice of approximation and architecture helps to improve the learning.

In Section~\ref{sec:background}, we present the problem. In Section~\ref{sec:sens}, we prove theoretical guarantees on the MFC problem under distribution approximation. We present the algorithm and three distribution approximations in Section~\ref{sec:method}. Experiments are provided in Section~\ref{sec:numer}. We conclude the paper in Section~\ref{sec:conclusion} and discuss differences with related works in Section~\ref{sec:relatedwork}.

\section{Background}
\label{sec:background}
We first introduce the notations to define the MFC problem with common noise. Let $T>0$ be a finite horizon. Let $\cX = \RR^d$ be the state space and $\cA = \RR^k$ be the action space, where $d$ and $k$ are two integers. For simplicity of presentation, we work on the whole space, although the algorithm could be extended to compact domains. We will denote by $\cP^2(\cX)$ the spaces of probability measures with bounded second moments on $\cX$, endowed with the Wasserstein-$2$ distance denoted by $\cW_2$ the Wasserstein-2 distance and defined for two distributions $\mu, \mu' \in \cP^2(\RR^d)$ as: $\cW_2(\mu,\mu') := \inf_{\gamma \in \Gamma(\mu,\mu')} \big(\int_{\RR^d \times \RR^d}  \|x-x'\|^2 d\gamma(x,x')\big)^{1/2}$, where $\Gamma(\mu,\mu')$ is the set of probability measures in $\cP^2(\RR^d\times \RR^d)$ with marginals $\mu$ and $\mu'$. In this work, as is common in the literature on MFGs and MFCs, we focus on deterministic feedback control functions, also simply called controls in the sequel. However, in contrast to most of the literature, we consider controls that are functions not only of time and the agent's state but also of the state distribution of the population. To be specific, a control is a function $v:  [0,T] \times \cX \times \cP^2(\cX) \to \cA$. Let $V>0$ be a constant and let us denote by $\VV$ the set of controls that are $V$-Lipschitz in all variables. Let $(W_t)_{0\leq t\leq T}$ and $(W^0_t)_{0\leq t\leq T}$ be $d$-dimensional independent Brownian motions defined on a complete
filtered probability space $(\Omega,\mathbb{F}=(\mathcal{F}_t)_{0\leq t\leq T},\mathbb{P})$. We
shall refer to $W$ as the idiosyncratic noise and to $W^0$ as the common noise. We denote by $\F^0_t$ the filtration generated by $W^0$ up to time $t$ and refer to it as the common filtration. For a stochastic process $X = (X_t)_{t \ge 0}$, we denote by $\mathrm{Law}(X_t)$ and $\mathrm{Law}(X_t|\F^0_t)$ respectively the law of $X_t$ and the conditional law of $X_t$ given $\F^0_t$. 

The initial MFC problem with common noise is formulated as follows, where $b: [0,T]\times\cX \times \cA \times \cP^2(\cX) \to \RR^d$ is a drift function, $\sigma$ and $\sigma_0\in\mathbb{R}$ are volatilities, $f: [0,T]\times\cX \times \cA \times \cP^2(\cX) \to \RR$ is a running cost function and $g: \cX \times \cP^2(\cX) \to \RR$ is a terminal cost function. When needed, $\mu_0 \in \cP^2(\cX)$ is an initial distribution. 

\noindent\textbf{Problem 1 (Original MFC formulation)}
\textit{
Minimize over $v\in\VV$ the total expected cost:
\begin{equation*}J(v):=\EE\bigg[\int_0^Tf\big(t,X^{v}_t,A^{v}_t,\mu^{v}_t\big)\d t+g\big(X^{v}_T,\mu^{v}_T\big)\bigg],
\end{equation*}
where $A^{v}_t = v(t,X^{v}_t,\mu^{v}_t)$, $\mu^{v}_t = \mathrm{Law}(X^{v}_t|\F^0_t)$, and $X^v_t$ solves: 
\begin{equation}
\label{eq:process-X-v-continuous}
\begin{cases}
    X^{v}_0 \sim \mu_0
    \\
    dX^{v}_t = b(t,X^{v}_t,A^{v}_t,\mu^{v}_t) dt + \sigma dW_t+\sigma_0dW_t^0, t \ge 0.
\end{cases}
\end{equation}
}

The above dynamics of $X$ involves the law of the process itself and is often referred to as McKean-Vlasov dynamics~\cite{mckean1966class,carmona2018probabilistic}. It is also common to include the law of $X$ only in the cost function, for example in applications to risk management~\cite{andersson2011maximum}. We refer to~\cite{bensoussan2013meanbook} for more background on MFC.

\noindent\textbf{Standing assumption.} We assume that $b, f$ and $g$ are Lipschitz in all their respective variables.

\begin{remark}\label{rmk:sigmaconstant}
In general, the diffusion coefficients $\sigma$ and $\sigma^0$ could depend on the state, distribution and control as well, but for simplicity, we will focus on the above setting with constant diffusions.
\end{remark}

\section{Sensitivity to distribution approximation}\label{sec:sens}

\subsection{Perturbed dynamics and perturbed problem}
From a numerical viewpoint, since in general we cannot represent the distribution $\mu^{v}_t$ exactly, we will replace it by an approximation of the measure. This leads to the following perturbed problem with a perturbed dynamics, which can formally be defined as:

\noindent\textbf{Problem 2 (MFC with approximate distribution)}
\textit{Minimize over $v\in\VV$ the total expected cost:
\begin{equation*}
    \tilde J(v)=\EE\bigg[\int_{0}^{T}f\big(t,\hat X^{v}_t, \hat A^{v}_t, \tilde \mu^{v}_t\big)\d t+g\big(\hat  X^{v}_T, \tilde \mu^{v}_T\big)\bigg].
\end{equation*}
subject to:
\begin{equation*}
    \d \hat X^{v}_{t} = b(t, \hat X^{v}_t,\hat A^{v}_t,\tilde \mu^{v}_t) \d t + \sigma dW_t+\sigma_0dW_t^0,
\end{equation*}
where $\hat A^{v}_t = v(t, \hat X^{v}_t,\tilde \mu^{v}_t)$ and $\tilde \mu^{v}_t \approx\hat \mu_t =\mathrm{Law}(\hat X^{v}_t|\F_t^0)$. 
}

We can view $\tilde\mu^v_t$ as a perturbed version of $\hat\mu^v_t$, which itself is close to $\mu^v_t$ under suitable assumptions. Our main motivation is that numerical computations, the distribution cannot be represented perfectly and we want to account for this approximation. However, this problem is applicable to other settings and we could imagine applications in scenarios where the distribution is only partially observable. 
Importantly, note that the perturbation affects not only the drift and the costs, but also the actions since the distribution is an input of the control $v$. Moreover, note that $\tilde\mu^v_t$ is not necessarily equal to $\hat\mu^v_t:=$ Law$(\hat X^v_t|\mathcal{F}_t^0)$. A natural example is numerical approximations for mean field problems, as we will discuss later. 

Our first result shows that if the approximation is good, then the optimal values are close.  
\begin{theorem}\label{Thm_mainclose}
    Let $v \in \VV$. If $\int_{0}^T\mathbb{E}[\cW_2(\tilde\mu_t,\hat\mu_t)^2]\d t\leq \delta$ in the perturbed problem, then 
    \begin{equation*}
    \Big|\inf_{v\in\VV} J(v)- \inf_{v\in\VV} \tilde J(v)\Big|\leq C\delta,
    \end{equation*}
for some constant $C>0$ depending only on the Lipschitz constants of $b, f$ and $g$, on $V$ and on $T$. Moreover, let $v^*\in\VV$ be some near optimal control for the perturbed problem satisfying $\tilde J(v^*)\leq \inf_{v\in\VV} \tilde J(v)+\epsilon$ for some $\epsilon>0$. Then $v^*$ is also a near optimal control for the original problem in the sense that
\begin{equation*}
    J(v^*)\leq \inf_{v\in\VV} J(v)+\epsilon+C\delta,
\end{equation*}
for some constant $C>0$ depending only on the Lipschitz constants of $b,f$, and $g$, on $V$, and on $T$. 
\end{theorem}
In particular, an optimal control for $\tilde J$ is an approximately optimal control for the original $J$, and the sub-optimality decreases as the approximation of the distribution improves. The proof relies on the propagation of the distribution approximation through the dynamics and the cost function. The proof, provided below, relies on Lemma~\ref{lem:1} in Appendix~\ref{app:proof}.
\begin{proof}
For any $v\in\VV$, let $\tilde b(t,x,\mu):=b(t,x,v(t,x,\mu),\mu)$. Then
\begin{equation*}
    \d X^{v}_t=\tilde b(t,X^{v}_t,\mu^v_t)\d t+\sigma\d W_t+\sigma_0 \d W^0_t,
\end{equation*}
and
\begin{equation*}
    \d \hat X^{v}_t=\tilde b(t,\hat X^{v}_t,\tilde\mu^v_t)\d t+\sigma \d W_t+\sigma_0\d W^0_t,
\end{equation*}
Using Lemma \ref{lem:1} in Appendix~\ref{app:proof}, we get
\begin{equation*}
    \E\big[(X^{v}_t-\hat X^{v}_t)^2\big]+\E\big[\cW_2(\mu^{v}_t,\hat\mu^{v}_t)^2\big] \leq C_{T,V}\delta
\end{equation*}
for some constant $C_{T,V}$ depending only on the Lipschitz constant of $b,$, on $V$ and on $T$. Set $f^*(t,x,\mu)=f(t,x,v(t,x,\mu),\mu)$. Then:
\begin{equation*}
    J(v)=\E\bigg[\int_0^Tf^*\big(t,X^{v}_t,\mu^{v}_t\big)\d t+g\big(X^{v}_T,\mu^{v}_T\big)\bigg],
\end{equation*}
and
\begin{equation*}
    \tilde{J}(v)=\E\bigg[\int_0^Tf^*\big(t,\hat X^{v}_t,\tilde\mu^{v}_t\big)\d t+g\big(\hat X^{v}_T,\tilde \mu^{v}_T\big)\bigg].
\end{equation*}
Therefore, by the Lipschitz continuity of $f$ and $g$, we get 
\begin{equation*}
   |J(v)-\tilde{J}(v)|\leq C_{T,\ell}\delta,
\end{equation*}
where the constant $C_{T,\ell}$ now may also depend on the Lipschitz constants of $f$ and $g$. 
Hence, in particular:
\begin{equation*}
\Big|\inf_{v\in\VV} J(v)- \inf_{v\in\VV} \tilde J(v)\Big|\leq C_{T,\ell}\delta.
\end{equation*}
and moreover, if $v^*\in \VV$ satisfies
\begin{equation*}
    \tilde J(v^*)\leq \inf_{v\in\VV} \tilde J(v)+\epsilon,
\end{equation*}
then we have
\begin{equation*}
    \begin{aligned}
    J(v^*)\leq &\tilde J(v^*)+C_{T,\ell}\delta\leq \inf_{v\in\VV} \tilde J(v)+\epsilon+C_{T,\ell}\delta
    \\
    \leq & \inf_{v\in\VV} J(v)+\epsilon+2C_{T,\ell}\delta,
    \end{aligned}
\end{equation*}
which completes the proof.
\end{proof}

\begin{remark}
The above theorem is about the stability of the problem with respect to the approximation, which did not exist in the literature, to the best of our knowledge. It is motivated by numerical applications, using controls that are functions of an approximate distribution. In this paper, we use $N$-particle systems as a building block, but the above theorem could also be used for other approximation methods. 
\end{remark}

Let us turn to the $N$-particle approximation and then propose three canonical approximation algorithms based on empirical distribution, the empirical moments, and empirical histogram.

\subsection{$N$-particle approximation}
In this subsection, we are going to use a finite population of $N$ particles to construct a suitable approximation of the distribution. Let us first state a fundamental result that underpins our particle-based approach. Below, $\tilde{b}(t,x,\mu)$ plays the role of $b(t,x,v(t,x,\mu),\mu))$. It is Lipschitz thanks to the fact that the controls $v \in \VV$ are Lipschitz.
\begin{proposition}\label{prop_n}
Consider the SDE: 
\begin{equation*}
    \d X_t=\tilde b(t,X_t,\mu_t)\d t+\sigma \d W_t+\sigma_0 \d W^0_t,
\end{equation*}
where $\mu_t=\mathrm{Law}(X_t|\F_t^0)$ which we assume to have a moment of order $q>4$. Consider the system:
\begin{equation*}
    \d X_t^i=\tilde b(t,X_t^i,\bar\mu_t^N)\d t+\sigma\d W_t^i+\sigma_0\d W^0_t,
\end{equation*}
where $X_0^i$ are i.i.d. with the same law as $X_0$, $W^i$ are independent Brownian motions, and $\bar \mu_t^N = \frac{1}{N}\sum_{i=1}^N \delta_{X_t^i}$ denotes the empirical distribution of states. Then for all $t \in [0,T]$, 
$$
\int_0^T\mathbb{E}[\cW_2(\bar \mu_t^N,\hat\mu_t)^2]\d t\leq C\delta_N,
$$ 
where 
\begin{align}\label{delta_N}
\delta_N=N^{-2/\max(d,4)}(1+\ln(N) \boldsymbol{1}_{d=4}),
\end{align}
for some constant $C$ independent of $N$.
\end{proposition}
In particular, we can use $\tilde\mu_t = \bar \mu_t^N$ in Theorem~\ref{Thm_mainclose}.

The proof of proposition~\ref{prop_n} utilizes \cite[Theorem 2.12]{carmona2018probabilistic2}; see Appendix~\ref{app:proof} for more details. 
Let us stress that the above statement can be applied to \emph{any} distribution approximation technique. To the best of our knowledge, such statements did not exist in the literature, probably because the question of learning population-dependent policies using distribution approximations have garnered interest only recently. Next, we will study three such approximations.  Now let us define the perturbed problem using $N$-particles.

\noindent\textbf{Problem 3 ($N$-particle control problem)} \textit{
Minimize over $v \in\VV$ the total expected social cost:
\begin{equation*}
    \begin{aligned}
    J^N(v)=\frac{1}{N}\sum_{i=1}^N\EE\bigg[\int_0^Tf\big(t,X^{i,v}_t,A^{i,v}_t, \mu^{N,v}_t\big)\d t
    +g\big(X^{i,v}_T,\mu^{N,v}_T)\big)\bigg],
    \end{aligned}
\end{equation*}
where $A^{i,v}_t=v(t,X^{i,v}_t,\mu^{N,v}_t)$, $i=1,\dots,N$, and $\mu^{N,v}_t=\frac1N\sum_{i=1}^N\delta_{X^{i,v}_t}$, 
subject to: $X^{i,v}_0\sim\mu_0$ i.i.d., and for $t \ge 0$, 
\begin{equation*}
    \d X^{i,v}_t=b\big(t,X^{i,v}_t,A^{i,v}_t, \mu^{N,v}_t\big)\d t+ \sigma\d W^i_t+\sigma^0\d W_t^0,
\end{equation*}
}
The approximation of McKean-Vlasov dynamics using a system of interacting particles is classical in the literature, see e.g.~\cite{fouque2020deep,carmona2022convergence,kumar2022well}. The novelty here is the class of controls, which are allowed to depend on the empirical population distribution in a generic (Lipschitz) way.

Combining Theorem \ref{Thm_mainclose} and Proposition \ref{prop_n}, we get:
\begin{corollary}\label{cor_mainclose}
    Let $v \in \VV$, and consider the $N$-particle control problem. We have
    \begin{equation*}
    \Big|\inf_{v\in\VV} J(v)- \inf_{v\in\VV} J^N(v)\Big|\leq C\delta_N
    \end{equation*}
for some constant $C>0$ depending only on the Lipschitz constants of $b,\VV$, and on $T$ and where $\delta_N$ is defined as in \eqref{delta_N} with $\tilde{b}(t,x,\mu)=b(t,x,v(t,x,\mu),\mu))$. Moreover, let $v^*\in\VV$ being some near optimal control for the perturbed problem satisfying $J^N(v^*)\leq \inf_{v\in\VV} J^N(v)+\epsilon$ for some $\epsilon>0$.  Then $v^*$ is also a near optimal control for the original problem in the sense that
\begin{equation*}
    J(v^*)\leq \inf_{v\in\VV} J(v)+\epsilon+C\delta_N,
\end{equation*}
for some constant $C>0$ depending only on the Lipschitz constants of $b,f$ and $g$, on $V$ and on $T$. 
\end{corollary}

This result provides a theoretical foundation for the algorithms we propose below using an empirical distribution. 

\subsection{Implementing empirical distribution}
In general, representing a distribution can be challenging. The controller may not know the whole information of the distribution, and can only achieve limited information through some mapping $\psi: \mathcal{P}^2(\mathbb{R}^d)\to \mathbb{R}^m$. A first natural example is to represent the distribution through a finite number of samples, i.e., we take the mapping $\psi(\mu)=(X_t^1, X_t^2, \dots, X_t^N)$ which gives an empirical approximation provided the output of $\psi$ is then given to a function that is symmetric with respect to its $N$ inputs. Two other natural examples are moments: $\psi(\mu) = \big(\mathbb{E}_\mu[X^k]\big)_{k=0,1,..,m}$, and the histogram: $\psi(\mu) = \boldsymbol{p}^\mu$, where $\boldsymbol{p}^\mu$ is the histogram constructed from the empirical distribution using $B$ bins of uniform size over an hypercube of side length $L$, and one extra bin for points outside this hypercube, detailed definition can be found in \cite{lecoutre1985l2}. Therefore, it is natural to restrict our admissible set $\VV$ to the limited version $\VV^{\psi}$ where
 $
    v(t,X_t,\mu_t)=\tilde v(t,X_t,\psi(\mu_t))
$ 
for some $\tilde v:[0,T]\times \mathcal{X}\times\mathbb{R}^m\to\mathcal{A}$. In this case, the closeness corollary (Corollary~\ref{cor_mainclose}) can be modified as follows.
\begin{corollary}\label{cor_lim}
    Let $v \in \VV$, and consider the $N$-particle problem under the admissible set $\VV^\psi$ with Lipschitz function $\psi$, we have 
    \begin{equation*}
    \Big|\inf_{v\in\VV^\psi} J(v)- \inf_{v\in\VV^\psi} J^N(v)\Big|\leq C\delta_N
    \end{equation*}
for some constant $C>0$ depending only on the Lipschitz constants of $b,\VV^\psi, \psi$, and on $T$. Moreover, if $v^*\in\VV^\psi$ satisfies $J^N(v^*)\leq \inf_{v\in\VV^\psi} J^N(v)+\epsilon$, then 
\begin{equation*}
    J(v^*)\leq \inf_{v\in\VV^\psi} J(v)+\epsilon+C\delta_N,
\end{equation*}
for some constant $C>0$ depending only on the Lipschitz constants of $b, f, g$ and $\psi$, on $V$, and on $T$. 
\end{corollary}
The last result quantifies to what extent an approximately optimal control for the perturbed problem is also approximately optimal for the original problem. 

\begin{remark}
\label{rem:info_gap}
In general, there exists a gap between the original problem $\inf_{v\in\VV} J(v)$ and the restricted problem $\inf_{v\in\VV^\psi} J(v)$ because of the change of the admissible set from $\VV$ to $\VV^\psi$, which will be shown in some of the examples at Section \ref{sec:numer}. It is natural to ask when we can fill in this information gap, which is an intriguing problem for future research directions.
\end{remark}

\begin{remark}\label{rmk_mom}
Moment functions are not Lipschitz with respect to $\cW_2$ distance. To be more precise, we should look at the truncated moments and the details are discussed in the Appendix \ref{app:proof}.
\end{remark}

This inspires us to design the algorithm using empirical distribution, empirical moments and empirical histogram and to verify the convergence results.

\section{Method}\label{sec:method}

In this section, we describe the main components of the method we propose, with several variants of implementation.

\vskip6pt
\noindent{\bf Neural network architectures. } 
We are looking for optimal controls that are functions of time, state and the state distribution of the population. Therefore, in our numerical method we replace control $\tilde v$ by a parameterized function $\tilde{v}_{\theta_1} : [0,T] \times \mathbb{R}^d \times \mathbb{R}^m \to \mathbb{R}^k$ with parameter $\theta_1$. Since neural networks are good at approximating nonlinear functions and since the dimension of state or distribution approximation could be possibly high, we utilize fully connected feed-forward neural networks (NNs) to approximate the optimal control.

In order to approximate the distribution input that is used in the optimal control neural network, we apply two steps. Firstly, we use 3 different approaches (moments, histogram and empirical distribution) to summarize the state distribution of particles. This step can be thought as applying a mapping $\psi : \mathbb{R}^{N\times d}\to \mathbb{R}^{l_1}$ to the states of $N$ particles. Then we use this as an input (after the possible necessary reshaping from $\mathbb{R}^{l_1}$ to $\mathbb{R}^{l_2}$) for our parameterized distribution embedding function $m_{\theta_2} : \mathbb{R}^{l_2} \to \mathbb{R}^m$ with parameter $\theta_2$. In order to approximate this parameterized function, we use different NN architectures: feed-forward NNs (FFNN) (if the distribution summary is constructed by moments, histogram and empirical distribution), convolutional NNs (CNN) (if the summary is given by a histogram) and symmetric NNs (SYM) (if the summary is given by the empirical distribution). Here, the dimensions $l_1$ and $l_2$ depend on both distribution summary methods and also the type of the NN. For example, if we use histogram approximation with FFNN with state dimension is equal to 2, then $l_1=\texttt{nbin}\times\texttt{nbin}$ and $l_2=\texttt{nbin}^2$ where \texttt{nbin} is the number of bins of the histogram for each dimension of the state. In summary, we will compare 5 different approximation methods for the distribution embedding: {\bf i)} FFNN with empirical approximation (\textit{emp}), {\bf ii)} FFNN with moments (\textit{mom}), {\bf iii)} FFNN with histogram (\textit{hist}), {\bf iv)} CNN with histogram (\textit{hist\_CNN}), {\bf v)} SYM with empirical approximation (\textit{emp\_SYM}).

The FFNN and the CNN architectures are well-known but the symmetric architectures are less standard. For the sake of clarity, let us explain in more details the symmetric neural network architecture that we use. Its form ensures that it is invariant with respect to permutations of the positions: let $x = (x^1,\dots,x^N)$ be the vector of positions for the $N$ particles, each of them in dimension $d$. In the notations of Section 4.1, the neural network is of the form:
    $$
        m_{\theta_2}(x) = \Phi_2\left( \frac{1}{N} \sum_{i=1}^N\Phi_1(x^i;\theta_{2,1});\theta_{2,2}\right),  \qquad \theta_2 = (\theta_{2,1},\theta_{2,2})
    $$
    where $\Phi_1(\cdot; \theta_{2,1}): \mathbb{R}^d \to \mathbb{R}^{d_I}$ is a neural network with parameters $\theta_{2,1}$ (in the implementation of ``empirical + SYM'', it is the 4 hidden layers  and $d_I = 100$), and $\Phi_2(\cdot; \theta_{2,2}): \mathbb{R}^{d_I} \to \mathbb{R}^{m}$ (in the implementation of ``empirical + SYM'', this is the output layer).

    In contrast, the empirical + FFNN architecture is of the form: 
    $$
        m_{\theta_2}(x) = \Phi\left( (x^1,\dots,x^N);\theta_{2}\right).
    $$
      where $\Phi(\cdot, \theta_2): \mathbb{R}^{N\times d} \to \mathbb{R}^{m}$ is a neural network with parameters $\theta_{2}$ (in the implementation of ``empirical +FFNN'', it is the 4 hidden layers and 1 output layer).

\vskip6pt
\noindent{\bf Monte Carlo simulation. }
In order to simulate the trajectory of $\boldsymbol X$, we will use discrete time dynamics with $N$ particles. We denote with $\llbracket N \rrbracket = \{1,\dots, N\}$ the set of indices of particles. Let $\mathcal{T}=\{0,\Delta t, 2\Delta t, \dots, n\Delta t =T\}$. We construct Monte Carlo trajectories, $(X_t^i)_{t\in \mathcal{T},\ i \in \llbracket N \rrbracket}$ given parameterized control function $\tilde{v}_{\theta_1}$ and the empirical distribution $\mu_t^N$ obtained by simulating the $N$ particles. After initialization of $X_0^i \sim \mu_0$,\footnote{In some of the examples, we implemented common initial randomness instead of common noise in the dynamics. In that case $X_0^i = \tilde{X}^i_0 + x_0$ with $\tilde{X}^i_0 \sim \mu_0 $ and where $x_0\sim\mu_0^0$ denotes the common initial randomness and we take $\sigma^0=0$.} the iterations continue until $t=T$. The discrete time updates are done by using the following Euler-Maruyama approximation of the continuous time dynamics:
\begin{equation}
\label{eq:X_discrete_dynamics_numerics}
X^i_{t+\Delta t} = X^i_t+b(t,X^i_t,\tilde{v}_{\theta_1}^{i},\mu^N_t) \Delta t
+\sigma\epsilon^i_t
+\sigma^0\epsilon_t^0,      
\end{equation}
where $\tilde{v}_{\theta_1}^{i} = \tilde{v}_{\theta_1}^{i}(t, X_t^i, m_{\theta_2}(\psi(\boldsymbol{X}_t)))$, $\boldsymbol{X}_t = (X_t^i)_{i \in \llbracket N \rrbracket}$ and $\mu_t^{N}= \frac{1}{N}\sum_{i=1}^N \delta_{X_t^{i}}$. $\epsilon^i_t \sim \mathcal{N}(0, \Delta t)$ and $\epsilon^0_t \sim \mathcal{N}(0, \Delta t)$ represent the idiosyncratic and common noises, respectively. The details of the Monte Carlo simulation and distribution embedding methods can be found in Algorithm~\ref{algo:sim-X}.

\begin{algorithm}[t]
\caption{Monte Carlo simulation of an interacting batch with distribution embedding \label{algo:sim-X}}
\textbf{Input:} number of particles $N$; time horizon $T$; time increments $\Delta t$; initial distribution $\mu_{0}$; initial common randomness $\mu_0^0$; control function $\tilde{v}$; distribution embedding function $m$; number of bins for histogram approximation \texttt{nbin}; number of moments \texttt{nmom}; type of the approximation \texttt{type}

\textbf{Output:} Approximate sample trajectories of $(\boldsymbol X)$ using \eqref{eq:X_discrete_dynamics_numerics}\vskip1mm
\begin{algorithmic}[1]
\STATE{Set $n=0, t=0, x_0 \sim \mu_0^0, \tilde{X}_0^i\sim \mu_0\ i.i.d., X_0^i =\tilde{X}_0^i +x_0, \forall i\in \llbracket N \rrbracket$ } \vskip1mm
\WHILE{$n\times \Delta t = t\leq T$}\vskip1mm
    
    \IF{\texttt{type}=moment}
    \STATE{Set $\psi(\boldsymbol X_t) = (\frac{1}{N}\sum_{i=1}^N X_t^i, \frac{1}{N}\sum_{i=1}^N (X_t^i)^2, \dots$\\$ \frac{1}{N}\sum_{i=1}^N (X_t^i)^{\texttt{nmom}})$}
    \ELSIF{\texttt{type}=histogram}
    \STATE{Set $\psi(\boldsymbol X_t) =$  \texttt{nbin}$^m$ dimensional tensor that counts the number of particles in each bin}
    \ELSIF{\texttt{type}=empirical}
    \STATE{Set $\psi(\boldsymbol X_t)  = (X_t^i)_{i \in \llbracket N \rrbracket}$}
\ENDIF    
    \STATE{Compute distribution embedding $m(\psi(\boldsymbol X_t))$}
    \vskip1mm
    \STATE{Set $\tilde{v}^i_{t} = \tilde{v}(t, X_t^i, m(\psi(\boldsymbol X_t)))$}
    \vskip1mm
    \STATE{Set 
    $\mu_t^N = \sum_{i=1}^{N}\delta_{X_t^i}$}
    \vskip1mm
    \STATE{Let $X^i_{t+\Delta t} = X^i_t + b(t, X^i_t, \tilde{v}^i_t, \mu_t^N) \Delta t +  \sigma \epsilon_t + \sigma^0\epsilon^0_t$}\vskip1mm
    \STATE{Set $n=n+1, t= t+\Delta t$}\vskip1mm
\ENDWHILE\vskip1mm

\STATE{\textbf{return:} $(X_t^i)_t=0,1,\dots, T,\ i \in \llbracket N \rrbracket$} \vskip1mm
\end{algorithmic}
\end{algorithm}

\vskip6pt
\noindent{\bf Training method. }
Our goal is now to minimize over $\theta = (\theta_1, \theta_2)$ the average cost of $M$ populations of size $N$: 
\begin{equation}
\begin{aligned}
    \mathbb{J}^N(\theta)=\frac{1}{M} \sum_{j=1}^M
    \bigg[\sum_{t\in \mathcal{T}}
    \frac{1}{N} \sum_{i=1}^{N}
    f\Big(t,X^{i,j, \theta}_t,
    \tilde{v}_{\theta_1},\mu^{N,j, \theta}_t\Big)\Delta t + \dfrac{1}{N}\sum_{i=1}^N g\big(X^{i,j, \theta}_T,\mu^{N,j, \theta}_T\big)\bigg],
\end{aligned}
\end{equation}
where $\tilde{v}_{\theta_1} = \tilde{v}_{\theta_1}\big(t, X_t^{i,j,\theta}, m_{\theta_2}(\psi(\boldsymbol X^{j,\theta}_t))\big)$, $\boldsymbol X^{j}_t = (X^{i,j}_t)_{i \in \llbracket N \rrbracket}$ is the particles' states in population $j$ and $\mu_t^{N,j}= \frac{1}{N}\sum_{i=1}^N \delta_{X_t^{i,j}}$ is the empirical state distribution in population $j$. Furthermore, $\tilde{v}_{\theta_1}(\cdot,\cdot,\cdot)$ and $m_{\theta_2}(\cdot)$ are neural networks to approximate the optimal control and distribution embedding. Here, $\psi(\cdot)$ is a distribution approximation method we are using (moments, histogram or empirical). 

In order to optimize over $\theta=(\theta_1, \theta_2)$, we  use Adam optimizer (Adaptive Moment Estimation algorithm) which is a variant of stochastic gradient descent. Instead of sampling $M$ populations of size $N$, we  sample one population of size $N$ at each iteration and minimize over the following cost
\begin{equation}
\begin{aligned}
\label{eq:cost_1pop_discrete_numerics}
\mathbb{J}_S^N(\theta)=\sum_{t\in \mathcal{T}}\frac{1}{N} \sum_{i=1}^{N}f\Big(t,X^{i, \theta}_t,\tilde{v}_{\theta_1},\mu^{N, \theta}_t\Big)\Delta t
     + \frac{1}{N}\sum_{i=1}^N g\big(X^{i, \theta}_T,\mu^{N, \theta}_T\big),
\end{aligned}
\end{equation}
where $\tilde{v}_{\theta_1} = \tilde{v}_{\theta_1}\big(t, X_t^{i,\theta}, m_{\theta_2}(\psi(\boldsymbol X^{\theta}_t))\big)$ and $\boldsymbol{X}^\theta_t = (X_t^{i, \theta})_{i \in \llbracket N \rrbracket}$. The detailed pseudo-code of our methods can be found in Algoritm~\ref{algo:SGD-MFC}. 

Remember that our goal is to learn an approximately optimal control and this problem is high dimensional because of using the distribution as an input. In contrast to the Master equation (see e.g.,~\cite{cardaliaguet2019master}), which is posed for all distributions, here we are concerned with the performance of the control on realistic sequences of distributions that arise due to randomness in the initial condition or the dynamics. For more details on the high dimensionality of our experiments, please refer to Appendix~\ref{app:algorithms}.

\begin{algorithm}[t]
\caption{Stochastic Gradient Descent for population-dependent MF controls with distribution embedding \label{algo:SGD-MFC}}
\textbf{Input:} Initial parameter $\theta_0$; number of iterations $K$; sequence $(\beta_k)_{k = 0, \dots, K-1}$ of learning rates; number of particles $N$

\textbf{Output:} Approximation of $\theta^*$ minimizing representative player's cost\vskip1mm
\begin{algorithmic}[1]
\FOR{$k=0,1,\dots,K-1$}\vskip1mm
    \STATE{Sample $(X^i_{t})_{t=0,\Delta t,\dots, T,\ i\in  \llbracket N \rrbracket}$} using Algorithm~\ref{algo:sim-X} with control function $\tilde{v}=\tilde{v}_{\theta_{k,1}}$, distribution embedding function $m=m_{\theta_{k,2}}$ and parameters: $N, T, \Delta t, \mu_0, \mu_0^0$, \texttt{type}, \texttt{nmom}, \texttt{nbin}\vskip1mm
    \STATE{Compute the gradient $\nabla \mathbb{J}^{N}_S(\theta_k)$ of the cost function defined in \eqref{eq:cost_1pop_discrete_numerics}}\vskip1mm
    \STATE{Set $\theta_{k+1} = \theta_k -\beta_k \nabla \mathbb{J}^{N}_S(\theta_k)$}\vskip1mm
\ENDFOR
\STATE{\textbf{return:} $\theta_K$}
\end{algorithmic}
\end{algorithm}

\section{Numerical experiments}\label{sec:numer}

\noindent{\bf Example 1: Systemic risk. }
First, we focus on the discrete time version of the systemic risk model with the common noise that is analyzed in~\cite{carmona2018probabilistic}. This model is introduced for modeling the borrowing and lending of banks. The state of the players i.e., the banks, is the logarithm of their cash reserve and the state of the representative player at time $t \in \mathcal{T}$ is denoted as $X_t$. The dynamics of the representative player's log-cash reserve is:
\begin{equation}
    X_{t+\Delta t} = X_t + \big[a(\bar X_t - X_t)+A_t\big]\Delta t + \sigma \overline{\epsilon}_t,\ X_0 \sim \mu_0,
\end{equation}
$\forall t\in \{0, \Delta t, \dots, (n-1)\Delta t\}$ and where $\overline{\epsilon}_t = \sqrt{1-\rho^2} \epsilon_t + \rho \epsilon_t^0,$ 
 for some $\rho\in[0,1]$. Here, $\epsilon_t \sim N(0, \Delta t)$ denotes the idiosyncratic noise and $\epsilon_t^0\sim \mathcal{N}(0,\Delta t)$ denotes the common noise. Furthermore, we assume $a$ is a positive constant. The control of the bank, $A_t$, is the lending and borrowing amounts at time $t$. The objective of the bank is to minimize over $A$:
\begin{equation}
\begin{aligned}
    &\mathbb{E}\Big[\sum_{t\in \mathcal{T}}\big[\dfrac{A_t^2}{2} - qA_t(\bar X_t- X_t) + \dfrac{\varepsilon}{2}(\bar X_t - X_t)^2\big]\Delta t + \dfrac{c}{2} (\bar X_T - X_T)^2\Big]
\end{aligned}
\end{equation}
where $\varepsilon, c, \lambda >0$ are constants, $\bar X_t =\int_{\mathbb{R}} x d\mu_t (x),$ for all $t\in \mathcal{T}$ where $\mu_t=\mathrm{Law}{X_t}$ and $\boldsymbol A = (A_t)_{t\in\mathcal{T}}$. Here $\varepsilon$ and $c$ balance the individual bank's behavior with the average behavior of the other banks. $q>0$ weighs the contribution of the components and helps to determine the sign of the control i.e., whether to borrow or to lend. We assume $q^2 \leq \varepsilon$ in order to guarantee the convexity of the running cost functions.

In Figure~\ref{fig:sysrisk} (left), we see that different distribution approximation approaches result in similar loss functions. In Figure~\ref{fig:sysrisk} (right), we compare our control vs. state results with the explicit solutions (cf. \cite{carmona2018probabilistic} for the explicit solutions) for the sanity check. We can see that our numerical approach is good at mimicking the explicit solutions and the approximation with the histogram is working well. In this example, the total distribution embedding input is of dimension $1000$ when using empirical approximation with FFNN (since $1000$ banks are simulated) which shows the high dimensionality of our problem. For further details on the input dimensions, please refer to Appendix~\ref{app:algorithms}.

\begin{figure}[h]
\centering
\begin{subfigure}[t]{0.42\textwidth}
\centering
    \includegraphics[width=1.1\linewidth]{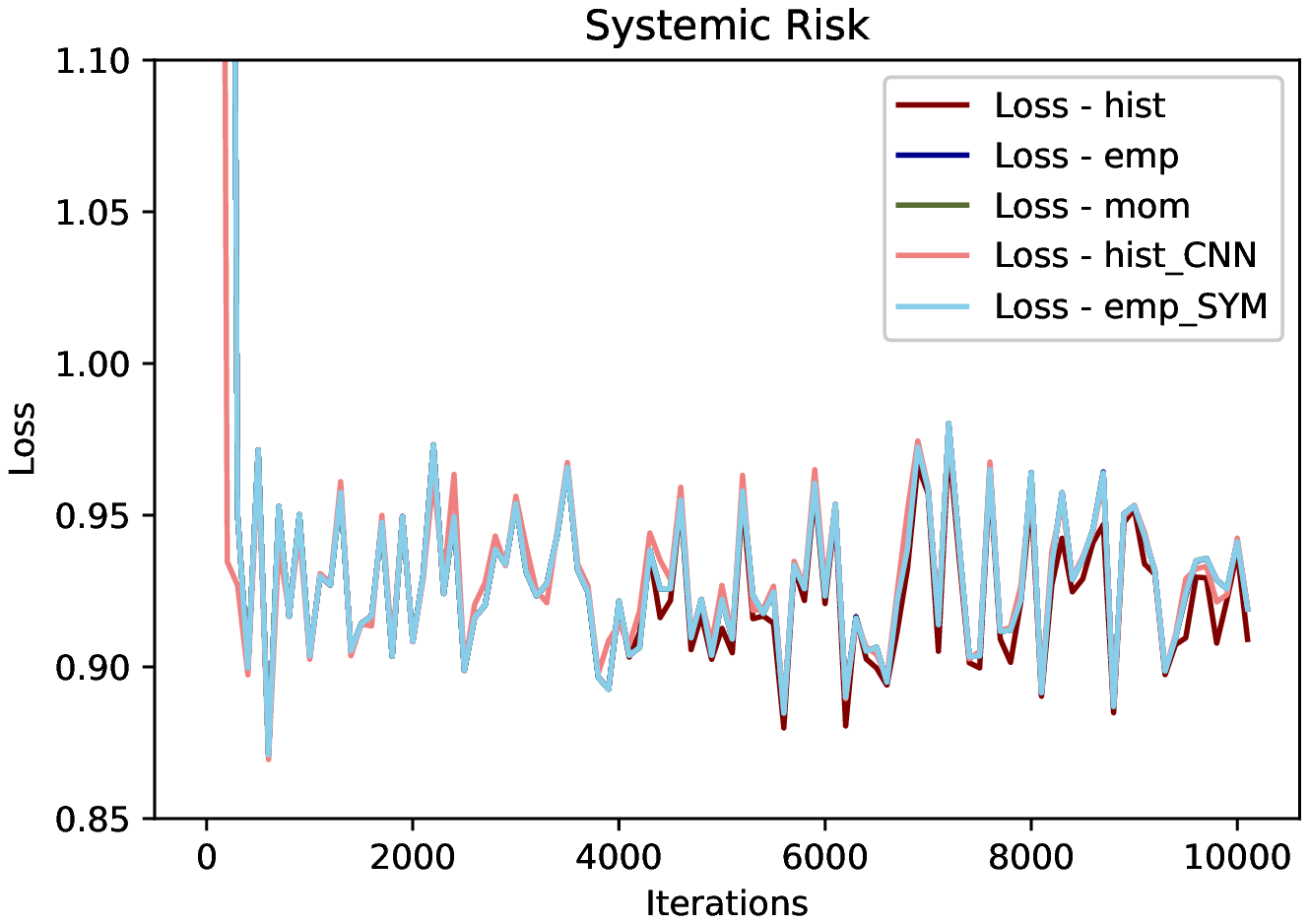}
\end{subfigure}
\hfill
\begin{subfigure}[t]{0.56\textwidth}
\centering
    \includegraphics[width=1.1\linewidth]{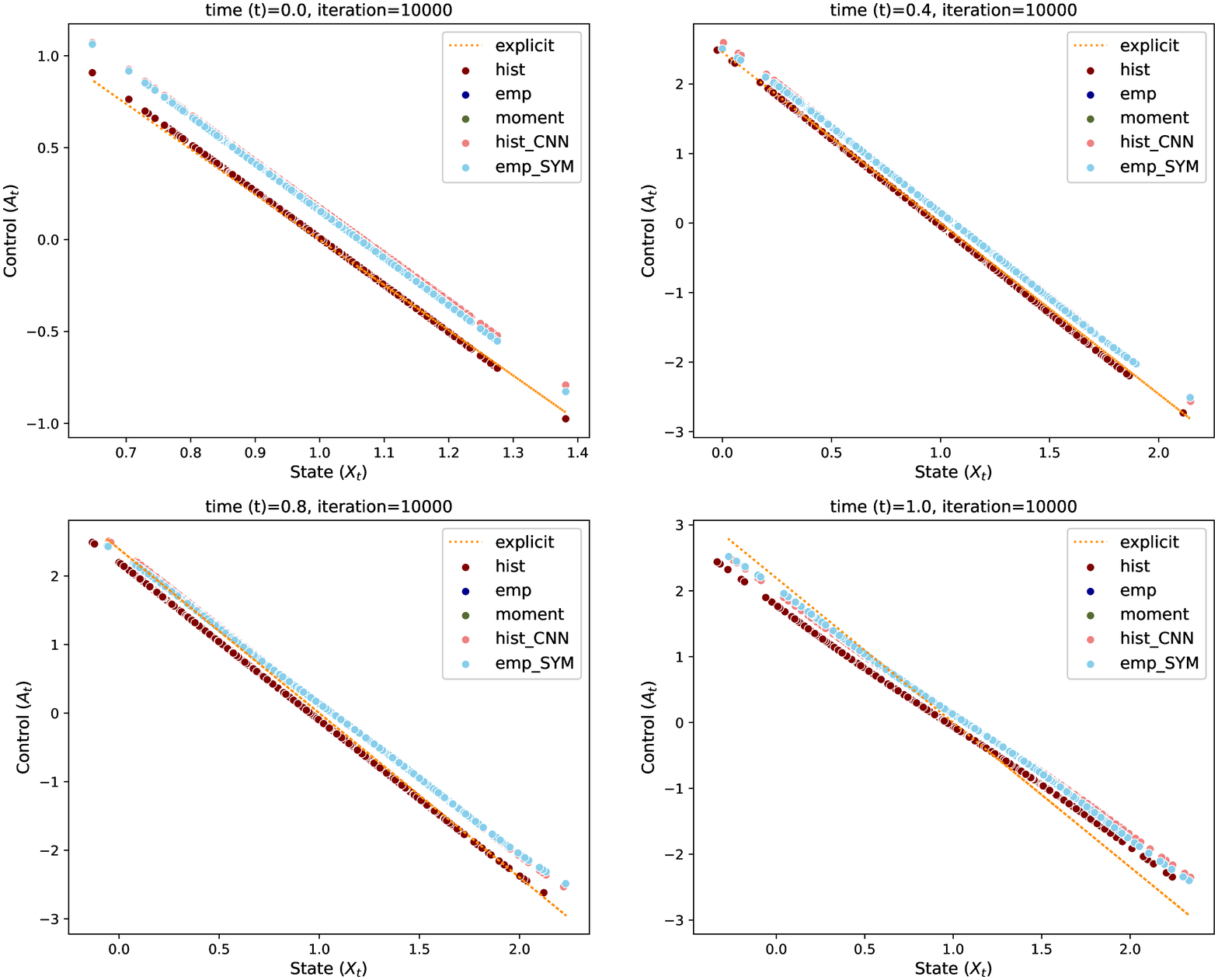}
\end{subfigure}
\caption{\small Loss comparison of different distribution approximations (left) and comparison of the numerical solutions with explicit solutions (right) in the systemic risk experiments.}
\label{fig:sysrisk}
\end{figure}

\vskip6pt
\noindent{\bf Example 2: Price impact. }
We extend the discrete time version of the price impact model given in ~\cite{carmona2018probabilistic} to incorporate trading of multiple stocks, which yields a 2D model. A representative trader controls its inventory for two different stocks, $X_t^1$ and $X_t^2$ by the rate of trading for each stock, $A_t^1$ and $A_t^2$. In this example, the common noise is at the initial condition of the dynamics. The dynamics of the representative trader's inventory is:
\begin{equation*}
    \begin{aligned}
        X_{t+\Delta t}^i = X^i_t + A_t^i \Delta t + \sigma_i \epsilon^i_t,\quad X_0^i= \tilde{X}^i_0 +x_0^i, \quad i \in \{1,2\},
    \end{aligned}
\end{equation*}
where $\epsilon_t^1 \sim \mathcal{N}(0,\Delta t)$ and $\epsilon_t^2\sim \mathcal{N}(0,\Delta t)$ are independent idiosyncratic noises, $\tilde{X}^i_0\sim \mu_0^i$ for all $i\in\{1,2\}$, and $x_0^1\sim \mathcal{N}(0, \sigma_1^0)$, $x_0^2\sim \mathcal{N}(0, \sigma_2^0)$ are common initial randomness. The price impact can be seen in the dynamics of the mid-price of Stock $i$ through the linear \textit{instantaneous market impact function}:
\begin{equation*}
    \begin{aligned}
        S_{t+\Delta t}^i &= S_t^i + h_i \bar A_t^i \Delta t + \sigma^i_0 \epsilon_t^{0,i}, \quad i \in \{1,2\},
    \end{aligned}
\end{equation*}
where $\epsilon_t^{0,i}\sim\mathcal{N}(0,\Delta t),\ i=\{1,2\}$ are the common noises which are independent from $\epsilon_t^1$ and $\epsilon_t^2$, and $h^i>0$ is the price impact. The interactions are coming through the controls' mean $\bar{A}_t^i$. Our numerical approach still can be adapted to such extended MFC. The amount of cash held by the representative trader at time $t$ is denoted by $K_t$ and it has the following dynamics:
\begin{equation*}
    K_{t+\Delta t} = K_t -[A_t^1 S_t^1 + A_t^2 S_t^2 + \frac{c_\alpha}{2}((A_t^1)^2+(A_t^2)^2)]\Delta t
\end{equation*}
where $c_\alpha>0$ is a constant coefficient. Here, for the cost of trading at the chosen rate, we used a quadratic cost which corresponds to flat order book.
The representative trader wants to maximize their expected wealth at the terminal time ($V_T = K_T + X^1_TS^1_T + X_T^2S_T^2$ where $T=n\Delta t$) and they are subject to liquidation constraints i.e., they want to minimize the shares held at each time $t\in \mathcal{T}$. Therefore, the representative player has the following cost, where $c_X, c_g>0$ and $\boldsymbol{A} = (A_t^1, A^2_t)_{t\in\mathcal{T}}$:
\begin{equation*}
\begin{aligned}
    \inf_{\boldsymbol A}  \sum_{i=1,2}\mathbb{E}\Bigg[ \sum_{t\in \mathcal{T}}\Big[&\dfrac{c_\alpha}{2}(A_t^i)^2 + \dfrac{c_X}{2} (X_t^i)^2
    - h_iX_t^i\bar{A}_t^i\Big] \Delta t 
    + \dfrac{c_g}{2} (X_T^i)^2\Bigg].
\end{aligned}
\end{equation*}

With Figure~\ref{fig:primpact} (left), we can compare the performance of different distribution approximation methods. We also added \textit{nodist} case where we do not use distribution as an input in our control approximation i.e., control is a state-dependent control instead of population-dependent control. We can see that empirical approximation with FFNN is good at minimizing the cost before the other methods. After around $7000$ iterations, empirical approximation with SYM and moments approximation with FFNN also improves the results further than not using any distribution approximation as an input in the optimal control, i.e., \textit{nodist} case. In Figure~\ref{fig:primpact} (right), we visualize 3 different distribution summary methods used: empirical distribution, moments, and histogram. The scatter plot represents the positions of the particles at a fixed time and visualizes the empirical approximation, the histograms on the axes represent the histogram approximation and the point that the horizontal and vertical green lines interact is the average of positions that represents the first moment approximation. In this example, the total distribution embedding input is of dimension $2*800$ when using empirical distribution with FFNN (since $800$ traders with a state dimension of $2$ are simulated) which shows the high dimensionality of our problem. 

\begin{figure}[h]
\centering
\begin{subfigure}[t]{0.5\textwidth}
\centering
    \includegraphics[width=\linewidth]{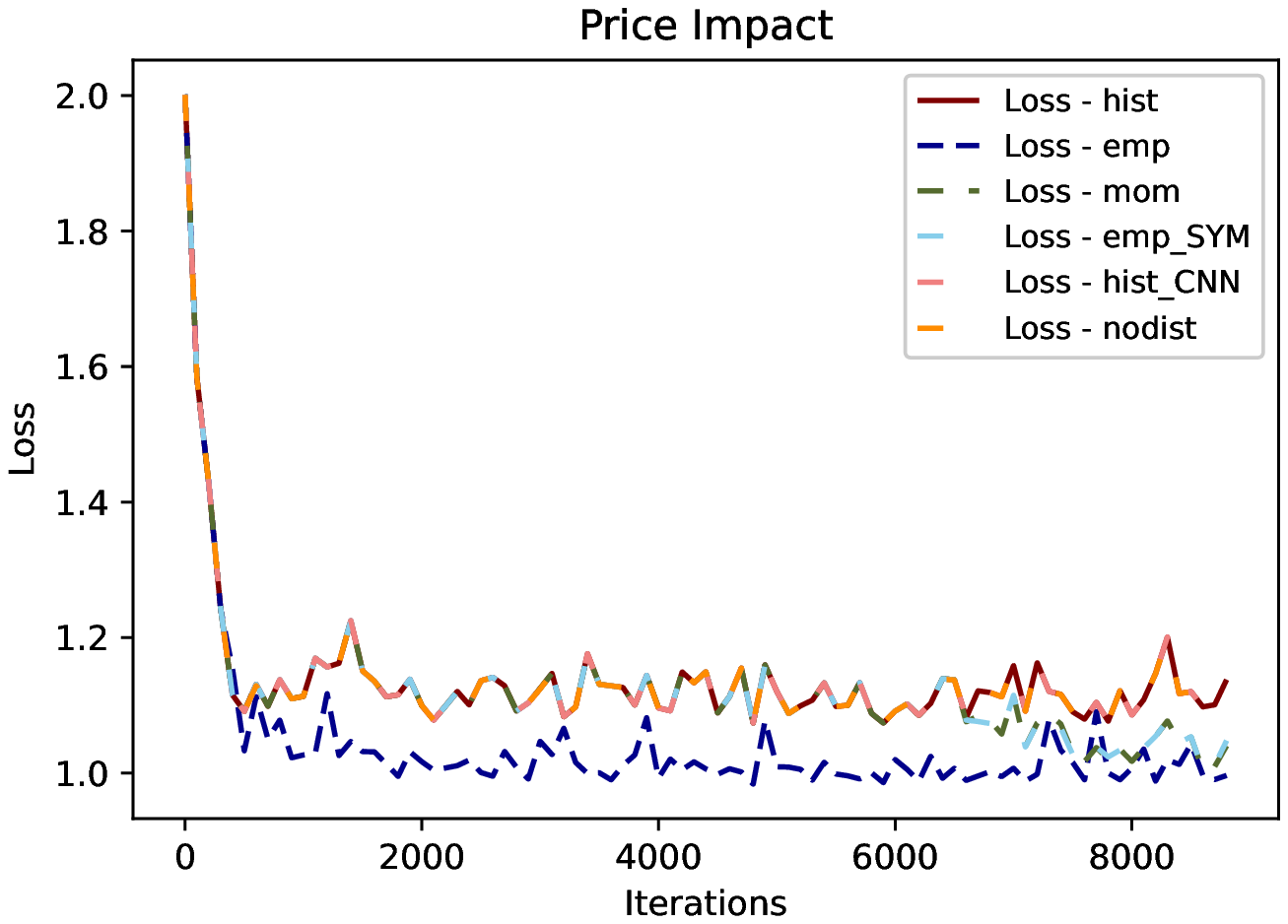}
\end{subfigure}
\hfill
\begin{subfigure}[t]{0.49\textwidth}
\centering
\includegraphics[width=0.9\linewidth]{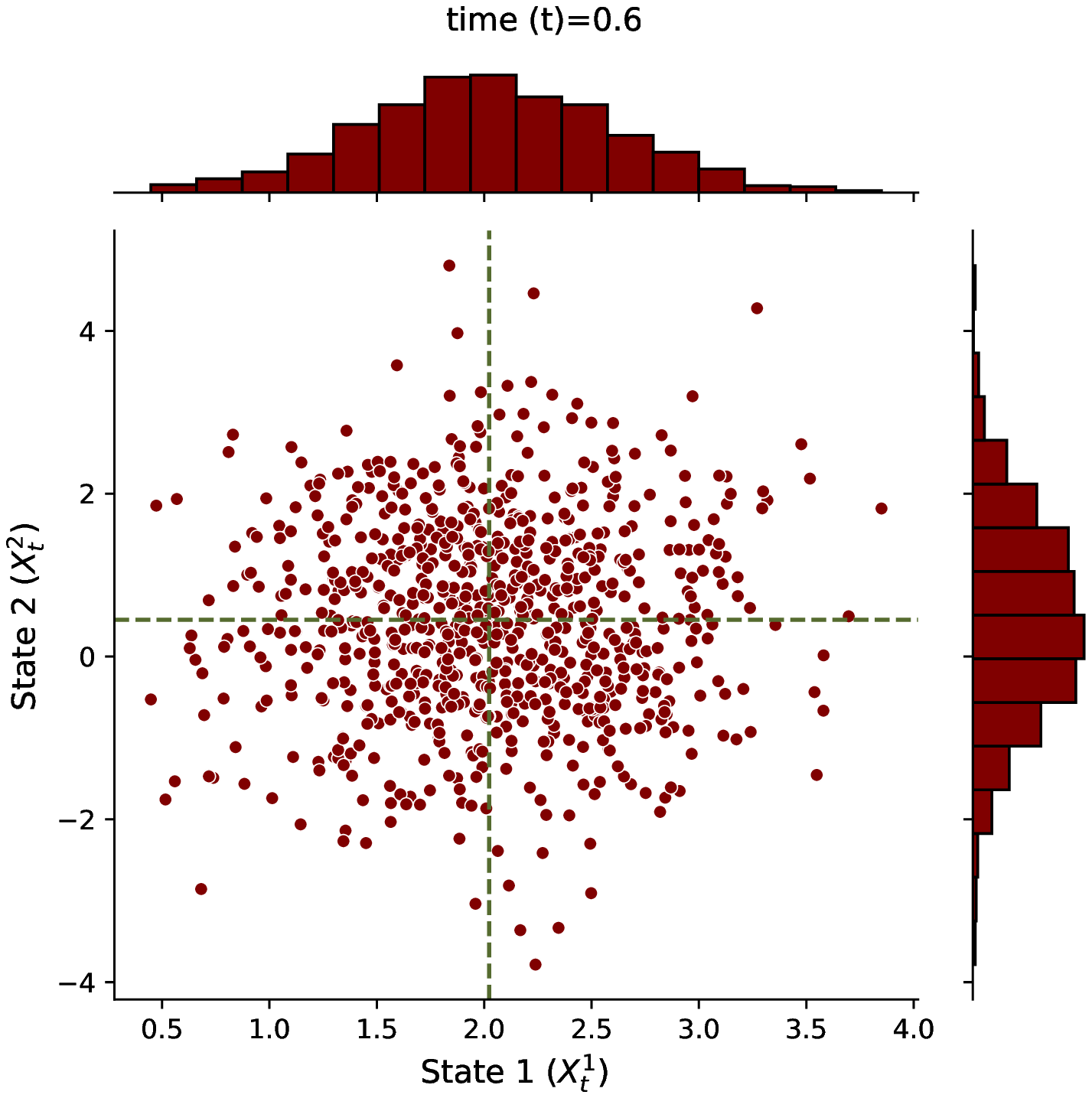}
\end{subfigure}
\caption{\small Loss comparison of different distribution approximations (left) and visualization of the state distribution approximations at time $t=0.6$ (right) in the price impact experiments.}
\label{fig:primpact}
\end{figure}

\noindent{\bf Example 3: Crowd motion with congestion. }
Finally, we look at a more complex crowd motion problem. The representative player controls her position at time $t$, $X_t$, by choosing her velocity, $A_t$. Therefore, we consider the following dynamics for the representative player:
\begin{equation}
    X_{t+\Delta t} = X_t + A_t \Delta_t + \sigma \epsilon_t,\quad X_0=\tilde{X}_0 +x_0,
\end{equation}
where $\epsilon_t\sim\mathcal{N}(0, \Delta t)$ is the idiosyncratic noise, $\tilde{X}_0\sim \mu_0$ is the player's state and $x_0\sim\mathcal{N}(0, \sigma_0)$ represents the common initial randomness. The cost function is as follows:
\begin{equation*}
\begin{aligned}
    \inf_{\boldsymbol A} \mathbb{E} \Bigg[ \sum_{t\in \mathcal{T}} \Big( \frac{1}{2}(c_0+\rho\star\mu_t)(X_t) |A_t|^2 + \ell(X_t, (c_0+\rho\star\mu_t)(X_t))\Big)\Delta t + g(X_T)\Bigg]
\end{aligned}    
\end{equation*}
where $c_0>0$ is a constant, $\rho$ is a smooth kernel (e.g., Gaussian), $\star$ denotes the convolution and $\boldsymbol A=(A_t)_{t\in\mathcal{T}}$. The first part of the running cost models congestion in the sense that it is more expensive to move in a crowded region (i.e., a region with high density) than in a non-crowded one. In this example, we take:
$
    \ell(x,m) = c_1 \|x - x_{\mathrm{target}}\|^2 + c_2 m, 
$ 
where $c_1,c_2>0$ are constants and $x_{\mathrm{target}} \in \mathbb{R}^d$ is a target position. Similarly, for $g$ we take 
$
    g(x) = c_3 \|x - x_{\mathrm{target}2}\|^2, 
$  
where $c_3>0$ is a constant and $x_{\mathrm{target}2} \in \mathbb{R}^d$ is another (or the same) target position.
We focus on a $d=2$ dimensional example. The parameters used in the experiment can be found in the Appendix~\ref{app:parameters}. In Figure~\ref{fig:crowdmotion} (left), we have the loss comparison for different approximation methods (including the \textit{nodist} case where the control is a state-dependent control instead of a population-dependent control). We can see that including distribution approximation as an input to our optimal control neural network improves the learning. In other words, population-dependent controls outperform the state-dependent control. We can also see that CNN with histogram approximation and SYM with empirical approximation are outperforming FFNN with histogram and empirical approximations. This shows that in this example using different NN architectures other than FFNN improves the learning. In Figure~\ref{fig:crowdmotion} (right), we compare the loss of FFNN with histogram approximations with different number of bins (2, 4 and 16). We can see that as the number of bins increases, the learning improves since more information is captured.

\begin{figure}[h]
\centering
\begin{subfigure}[t]{0.49\textwidth}
    \centering
    \includegraphics[width=\linewidth]{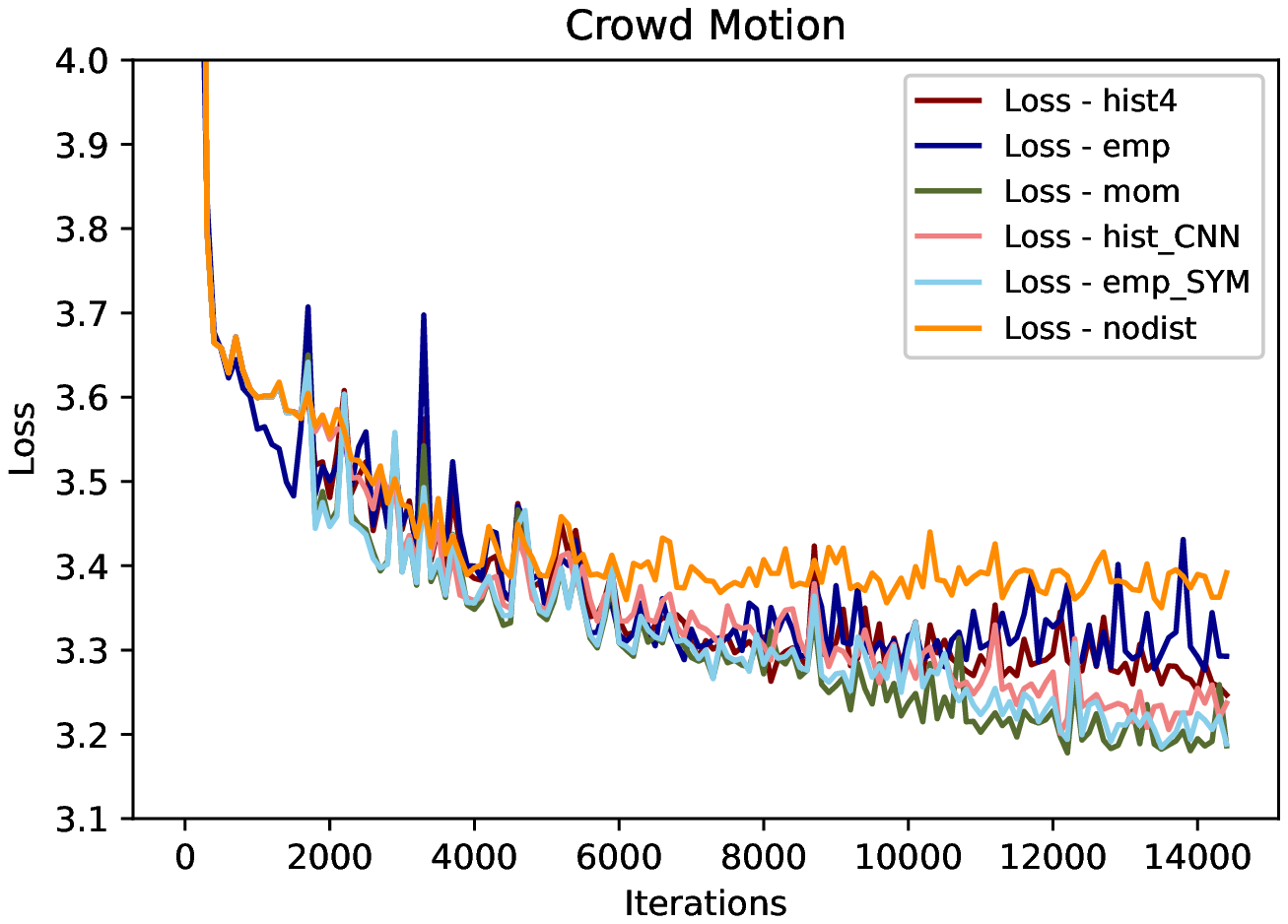}
\end{subfigure}
\hfill
\begin{subfigure}[t]{0.49\textwidth}
    \centering
    \includegraphics[width=\linewidth]{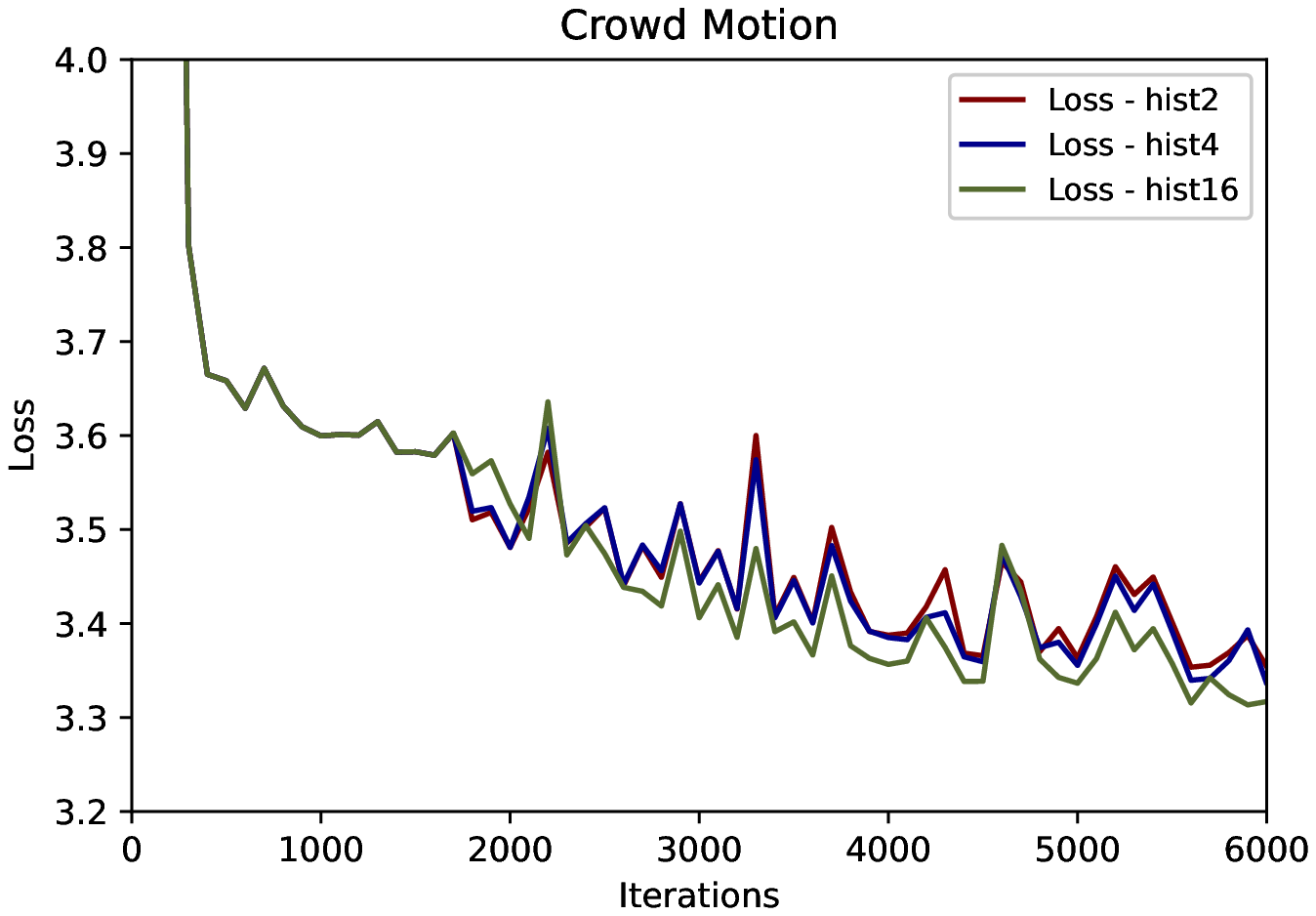}
\end{subfigure}
    \caption{\small Loss comparison of different distribution approximations (left) and loss comparison of different histogram bin number (right) in the crowd motion experiments.}
    \label{fig:crowdmotion}
\end{figure}

\section{Conclusion and limitations}
\label{sec:conclusion}
\noindent{\bf Conclusion. }
In this paper, we propose several algorithms to seek for optimal population-dependent controls, for solving MFC problems with common noise. Analysis of the convergence for the proposed algorithms and more general algorithms is provided. The effectiveness of the algorithms is justified by three concrete applications. In our first example, we look at an explicitly solvable problem and compare the results of our numerical approach with the explicit solutions to sanity check. In the second experiment, we look at a more complex models with increased dimensions for the states and controls. In this example, we show that population-dependent controls outperform state dependent controls. In our last example, we look at a complex crowd motion model beyond linear quadratic setup. We show that again population-dependent controls outperform state dependent controls and furthermore, the different NN architectures depending on the distribution approximation improve learning. For example, using symmetric NN with empirical approximation improves the results further than using FFNN with empirical approximation and using convolutional NN with histogram approximation improves the results further than using FFNN with histogram approximation.
\vskip6pt
\noindent{\bf Limitations. } For future work, we plan to close the gap between the original problem $\inf_{v\in \VV} J(v)$ and the restricted problem after the change of the admissible set from $\VV$ to $\VV^{\psi}$ $\inf_{v\in \VV^{\psi}} J(v)$. In other words, we plan to show the convergence of the optimal control of the restricted problem to the optimal control of the restricted problem. It would also be interesting to analyze theoretically and numerically the sample complexity both in terms of the size $N$ of one population and in terms of the number of population-wide samples $M$.

\section{Related work}
\label{sec:relatedwork}

Our algorithms build upon the method proposed in~\cite{carmona2022convergence}, which solves an MFC problem by training a neural network control using Monte Carlo simulations for a population of particles. However, in their work, the analysis is done for controls which are functions of the individual state only (i.e., state-dependent controls). 
In the experiments, it is shown that the method can handle simple forms of common noise (e.g., in linear-quadratic problems or problems in which the common noise realizations have a finite number of possible values), but it cannot directly handle general forms of population-dependence. \cite{min2021signatured} developed a fictitious play method for MFGs with common noise using signatures, which can cover cases where the interactions are through moments but not general forms of dependence on the mean field. \cite{perrin2022generalization} solve MFGs by learning a neural network for the Q-function (from which the policy can be deduced) taking as input a histogram representing the population distribution. However, those approaches are specific to finite-state problems (or histogram-based approximations) and their examples do not include common noise. \cite{germain2022deepsets} use a population-based algorithm and symmetric neural networks to solve the dynamic programming equation arising in MFC problems. One drawback of this approach is that the trajectory of the mean-field flow is not known when one uses a backward induction scheme. So their method requires learning the solution over many distributions that will not be actually useful. Our approach avoids this problem by simulating trajectories in a forward fashion, so that the neural network is mostly trained on relevant distributions. Distribution approximation through neural networks and its application to solve MFC problems have also been a focus of the recent works~\cite{pham2023meanfield,pham2022meanfield}, respectively. The main differences of these works from our work is three-fold: (1) Our main motivation is to treat generic MFC with common noise in continuous spaces, which is not covered by these works; (2) Our theoretical analysis takes into account the optimal value function, while Theorems 2.1 and 2.2 in~\cite{pham2023meanfield} deal with function approximation but not the optimal control aspects; (3) We provide multi-dimensional numerical examples while \cite{pham2022meanfield} has only 1D state space examples, and the extension from 1 to 2 state dimensions is (numerically) not trivial due to the combination of distribution approximation and NN controls. This creates a vast increase in the dimension of the NN inputs. For example, for the feedforward neural network (FFNN) with empirical approximation the input dimension goes from $N$ to $2N$. This also motivates us to use different NN architectures such as CNN and we showed that in the crowd motion experiment, CNN with histograms outperformed FFNN with histograms.\cite{carmonalaurieretan2019model,gu2021meanfieldcontrolsQ} propose RL algorithms for discrete time MFC problems through the lens of mean field Markov decision processes (MFMDP) in the infinite horizon discounted setting. Our approach can tackle continuous time problems with time-dependent controls, which is generally more challenging. 
The question of model-free methods has also received a growing interest in the context in mean field games, see e.g.~\cite{lauriere2022learning} for a recent overview. \cite{guo2019learning} proposes a Q-learning algorithm with Boltzmann policy with analysis of convergence property and computational complexity and applies to multi-agent reinforcement learning problem. The solution notion is different since it is a Nash equilibrium, which is different from the social optimum we study in this work.

\nocite{lecoutre1985l2}

{

}

\clearpage
\appendix
\section{Proof of theoretical results in Section \ref{sec:sens}}\label{app:proof}
Let us first present the following lemma which is fundamental to the proof of results in Section \ref{sec:sens}.
\begin{lemma}[Closeness of processes under perturbation / approximation of the law]\label{lem:1}
Suppose for $\tilde{b}$, $\sigma$, $\sigma_0$, we have
\begin{equation*}
    \begin{aligned}
    |\tilde{b}(t,x,\mu)-\tilde{b}(t,y,\nu)|\leq C\big(|x-y|+\cW_2(\mu,\nu)\big),\;
    \\
    |\sigma(t,x)-\sigma(t,y)|+|\sigma_0(t,x)-\sigma_0(t,y)|\leq C|x-y|,
    \end{aligned}
\end{equation*}
for some constant $C >0$.
Consider the SDE:
\begin{equation*}
    \d X_t=\tilde{b}(t,X_t,\mu_t)\d t+\sigma(t,X_t) \d W_t+\sigma_0 (t,X_t)\d W^0_t,
\end{equation*}
where $\mu_t=\mathrm{Law}(X_t|\F^0)$ and the perturbed SDE,
\begin{equation*}
    \d \hat X_t=\tilde{b}(t,\hat X_t,\tilde\mu_t)\d t+\sigma(t,\hat X_t) \d W_t+\sigma_0(t,\hat X_t) \d W^0_t,
\end{equation*}
where $\tilde\mu_t$ is a proper approximation of $\hat\mu_t=\mathrm{Law}(\hat X_t|\F^0)$ in the sense that $\int_0^t\mathbb{E}[\cW_2(\tilde\mu,\hat\mu)^2]\d t\leq \delta$ almost surely.
Then, we have
\begin{equation*}
    \E\big[(X_t-\hat X_t)^2\big]+\E\big[\cW_2(\mu_t,\hat\mu_t)^2\big] \leq \tilde C\delta,  \qquad t \in [0,T],
\end{equation*}
 for some constant $\tilde C>0$ depending only on the Lipschitz constants of $\tilde{b}, \sigma$ and $\sigma_0$ and on $T$. 
\end{lemma}
\begin{proof}

Using Ito's formula, we have
\begin{equation*}
    \begin{aligned}
    \d (X_t-\hat X_t)^2
    =&\Big(2(X_t-\hat X_t)\big(\tilde{b}(t,X_t,\mu_t)-\tilde{b}(t,\hat X_t,\tilde\mu_t)\big)
    \\
    &+\big(\sigma(t,X_t)-\sigma(t,\hat X_t)\big)^2+\big(\sigma_0(t,X_t)-\sigma_0(t,\hat X_t)\big)^2\Big) \d t
    \\
    &+2(X_t-\hat X_t)\big(\sigma(t,X_t)-\sigma(t,\hat X_t)\big)\d W_t
    \\
    &+2(X_t-\hat X_t)\big(\sigma_0(t,X_t)-\sigma_0(t,\hat X_t)\big)\d W^0_t.
    \end{aligned}
\end{equation*}

Taking expectation and using the assumptions on $\tilde{b}$, $\sigma$ and $\sigma_0$, we have
\begin{equation*}
    \begin{aligned}
    \d \E\big[(X_t-\hat X_t)^2\big]
    \leq\Big((2C+2C^2+Cc)\E\big[(X_t-\hat X_t)^2\big]
    +\frac{C}{c}\E\big[\cW_2(\mu_t,\tilde\mu_t)^2\big]\Big)\d t,
    \end{aligned}
\end{equation*}
for any constant $c>0$. Using Gr\"onwall's lemma, we get
\begin{equation*}
    \begin{aligned}
    \E\big[(X_t-\hat X_t)^2\big]
    \leq& \frac{Ce^{t(2C+2C^2+Cc)}}c\int_0^t\E\big[\cW_2(\mu_s,\tilde\mu_s)^2\big]\d s
    \\
    \leq&\frac{Ce^{t(2C+2C^2+Cc)}}c\int_0^t\E\big[\cW_2(\mu_s,\hat\mu_s)^2\big]\d s
    + \delta\frac{ tCe^{t(2C+2C^2+Cc)}}{c}.
    \end{aligned}
\end{equation*}
Using tower property, we have
\begin{equation}\label{equ1}
    \begin{aligned}
   \E\big[\cW_2(\mu_t,\hat\mu_t)^2\big] 
   \leq& \E\Big[\E\big[(X_t-\hat X_t)^2\big|\F^0\big]\Big]=\E\big[(X_t-\hat X_t)^2\big]
   \\
   \leq&\frac{Ce^{t(2C+2C^2+Cc)}}c\int_0^t\E\big[\cW_2(\mu_s,\hat\mu_s)^2\big]\d s 
+ \delta\frac{ tCe^{t(2C+2C^2+Cc)}}{c}.
   \end{aligned}
\end{equation}
Using Gr\"onwall's lemma again, we have
\begin{equation*}
    \E\big[\cW_2(\mu_t,\hat\mu_t)^2\big]\leq C_T \delta,
\end{equation*}
for some constant $C_T$. Now plugging back into \eqref{equ1}, and using the notation $C_T$ as a general constant, we get
\begin{equation*}
    \E\big[(X_t-\hat X_t)^2\big]\leq C_T \delta.
\end{equation*}
\end{proof}

\noindent\textbf{More comments on Remark \ref{rmk:sigmaconstant}:} In our model, $\sigma$ and $\sigma^0$ are taken as constants for the simplicity in presentation. Thanks to the above lemma, we can see that the model can be directly extended to the version where $\sigma$ and $\sigma^0$ are functions of time and state. It can be extended to the most general formulation (i.e., depending also on distribution).

\vskip12pt
\noindent\textbf{Proof of Proposition \ref{prop_n}.} 
This is a corollary of  Theorem 2.12 in \cite{carmona2018probabilistic2} that provides a bound on two terms, which combined, yield our result. More precisely, let us consider the left-hand side of their bound.  The second term gives, in the notation of our Proposition 3.2, $\cW_2(\bar\mu^N_t, \mu_t)$. In addition, the first term is, in our notation, $\mathbb{E}[\sup_{t} |X^i_t - X_t|^2]$ provided we take $W = W^i$. This term gives an upper bound on $\cW_2(\mathrm{Law}(X^i_t|\mathcal{F}^0_t), \mathrm{Law}(X_t|\mathcal{F}^0_t))^2 = \cW_2(\hat{\mu}_t, \mu_t)^2$, since for two random variables $Y$ and $Z$, we have $\cW_2(Y,Z) \le \mathbb{E}[|Y-Z|^2]^{1/2}$. By triangle inequality, we obtain the bound on $\cW_2(\bar\mu^N_t, \hat{\mu}_t)$.
\qed
\vskip12pt
\noindent\textbf{Proof of Corollary \ref{cor_mainclose} and \ref{cor_lim}.} Note that for any individual process $X^{i,v}$,
\begin{equation*}
    \d X^{i,v}_t=b\big(t,X^{i,v}_t,A^{i,v}_t, \mu^{N,v}_t\big)\d t+ \sigma\d W^i_t+\sigma_0\d W_t^0.
\end{equation*}
It is a perturbed process of the original process $X^{v}$ given in~\eqref{eq:process-X-v-continuous}. 
It satisfies the bound:
$$
\int_0^T\mathbb{E}[\cW_2(\text{Law}(X_t^{i,v}|\F^0),\mu^{N,v}_t)^2]\d t\leq CT\delta_N,
$$
for $i=1,\dots,N$ because of Proposition \ref{prop_n}. Then, we are able to combining Theorem \ref{Thm_mainclose} and Proposition \ref{prop_n} and moreover, in Corollary \ref{cor_lim}, for any $v\in\VV^\psi$, function $\tilde b$ defined in the proof of Theorem \ref{Thm_mainclose} is still Lipschitz. Therefore, we obtain both Corollary \ref{cor_mainclose} and \ref{cor_lim}.\qed
\vskip 12pt
\noindent\textbf{More explanation of Remark \ref{rmk_mom}. } To obtain the Lipschitz property of the moment functions, we should actually use the truncated moments in practice which will not change the results significantly. Therefore, the moment functions we consider become $\mathbb{E}\big[X^k\mathbbm{1}_{\{|X|\leq M\}}\big]$ for some $M$. This quantity is a Lipschitz function of the distribution of $X$ because, if $X$ and $Y$ are two random variables with distribution $\mu_X$ and $\mu_Y$, respectively, we have:
\begin{equation*}
\begin{aligned}
     \Big|\mathbb{E}\big[X^k\mathbbm{1}_{\{|X|\leq M\}}\big]-\mathbb{E}\big[Y^k\mathbbm{1}_{\{|Y|\leq M\}}\big]\Big|&\leq C_M\E[|X-Y|]\\
     &\leq C_M\E[|X-Y|^2]^{\frac12}\\
     &\leq C_M\cW_2(\mu_X,\mu_Y).   
\end{aligned}
\end{equation*}
\clearpage

\section{Further neural network implementation details}
\label{app:algorithms}

 We added below the neural network architectures that are used in every experiment (with every approximation type):
 \vskip12pt
        \noindent\textbf{Systemic Risk:}
        \vskip12pt
            \begin{center}
            \begin{tabular}{ |l |c| }
            \hline
             \textbf{Experiment Type} & \textbf{Distribution Embedding NN Arhitecture}  \\ \hline
             histogram + FFNN & 4 hidden layers of 100 nodes (sigmoid) + 1 output layer of 5 nodes (linear)  \\  \hline
             empirical + FFNN & 4 hidden layers of 100 nodes (sigmoid) + 1 output layer of 5 nodes (linear)   \\  \hline 
             moment + FFNN & 4 hidden layers of 100 nodes (sigmoid) + 1 output layer of 5 nodes (linear)   \\  \hline 
             & 3 one-dim. convolutional layers with (1x8), (1x4), (1x2) kernels (sigmoid)+\\  
             histogram + CNN &   1 flattening + 1 hidden layer of 100 nodes (sigmoid) +\\
             &1 output layer of 5 nodes (linear)\\  \hline 
             empirical + SYM & 4 hidden layers of 100 nodes (sigmoid) + 1 output layer of 5 nodes (linear)  \\  \hline 
            \end{tabular}
            \end{center}
        where FFNN is feedforward neural networks, CNN is convolutional neural network, SYM is symmetric neural network. The neural network architecture for control approximation is \emph{common to all the distribution approximation types} and is as follows: 4 hidden layers of 100 nodes (sigmoid) + 1 output layer of 1 node (linear).\\

       \noindent \textbf{Price Impact and Crowd Motion:}
            \begin{center}
            \begin{tabular}{ |l |c| }
            \hline
             \textbf{Experiment Type} & \textbf{Distribution Embedding NN Arhitecture}  \\ \hline
             histogram + FFNN & 4 hidden layers of 100 nodes (sigmoid) +1 output layer of 5 nodes (linear)  \\  \hline
             empirical + FFNN & 4 hidden layers of 100 nodes (sigmoid) + 1 output layer of 5 nodes (linear)   \\  \hline 
             moment + FFNN & 4 hidden layers of 100 nodes (sigmoid) + 1 output layer of 5 nodes (linear)   \\  \hline 
             & 3 two-dim. convolutional layers with (8x8), (4x4), (2x2) kernels (sigmoid) + \\ 
             histogram + CNN &  1 flattening +  1 hidden layer of 100 nodes (sigmoid) +\\
             &1 output layer of 5 nodes (linear)\\  \hline 
             empirical + SYM & 4 hidden layers of 100 nodes (sigmoid) + 1 output layer of 5 nodes (linear)  \\  \hline 
            \end{tabular}
            \end{center}
        We approximate control 1 and control 2 by using 2 different neural networks with the same architecture: 4 hidden layers of 100 nodes (sigmoid) + 1 output layer of 1 node (linear).
\vskip12pt

\noindent \textbf{Further information on ``empirical + SYM" distribution embedding NN architecture:}

The symmetric neural network architecture that we use is of the following form, which ensures that it is invariant with respect to permutations of the positions: let $x = (x^1,\dots,x^N)$ be the vector of positions for the $N$ particles, each of them in dimension $d$. In the notations of Section 4.1, the neural network is of the form:
    $$
        m_{\theta_2}(x) = \Phi_2\left( \frac{1}{N} \sum_{i=1}^N\Phi_1(x^i;\theta_{2,1});\theta_{2,2}\right),  \qquad \theta_2 = (\theta_{2,1},\theta_{2,2})
    $$
    where $\Phi_1(\cdot; \theta_{2,1}): \mathbb{R}^d \to \mathbb{R}^{d_I}$ is a neural network with parameters $\theta_{2,1}$ (in the implementation of ``empirical + SYM'', it is the 4 hidden layers  and $d_I = 100$), and $\Phi_2(\cdot; \theta_{2,2}): \mathbb{R}^{d_I} \to \mathbb{R}^{m}$ (in the implementation of ``empirical + SYM'', this is the output layer).

    In contrast, the empirical + FFNN architecture is of the form: 
    $$
        m_{\theta_2}(x) = \Phi\left( (x^1,\dots,x^N);\theta_{2}\right).
    $$
      where $\Phi(\cdot, \theta_2): \mathbb{R}^{N\times d} \to \mathbb{R}^{m}$ is a neural network with parameters $\theta_{2}$ (in the implementation of ``empirical +FFNN'', it is the 4 hidden layers and 1 output layer).

\vskip12pt
\noindent \textbf{Details on the dimensionality of the problem:}

The training and validation set sizes in the experiments can be found below:

\begin{center}
   {\small
\begin{tabular}{ |l |c| c |}
\hline
&\textbf{Training set}& \textbf{Validation set} \\ \hline
\textbf{Systemic Risk} & 1 population of size 1000 for each iteration & 1 population of size 1000 \\ \hline 
\textbf{Price Impact} & 1 population of size 800 for each iteration & Avg. loss over 5 populations of size 800 \\ \hline
\textbf{Crowd Motion} & 1 population of size 800 for each iteration & Avg. loss over 5 populations of size 800 \\ \hline
\end{tabular}} 
\end{center}

The input dimension details for the Distribution Embedding NN for each experiment can be found below.

\begin{center}{\small
\begin{tabular}{ |l |c| c| c| }
\hline
    \textbf{Experiment Type} & \textbf{Systemic Risk} & \textbf{Price Impact} & \textbf{Crowd Motion}  \\ \hline
 histogram + FFNN &  $5$ & $256$ & $16$\\  \hline
 empirical + FFNN &  $1000$ & $1600$ & $1600$\\  \hline 
 moment + FFNN &    $1$ & $2$ &$2$\\  \hline 
 histogram + CNN &  $32$ & $16 \times 16$& $16\times 16$\\ \hline 
 empirical + SYM &   $1$ & $2$& $2$\\  \hline 
\end{tabular}}
\end{center}

For all types of the distribution embedding NNs, the output dimension is set to 5 and it is inputted in the control approximation NNs by concatenating it with the time and state of the particle. Therefore input dimensions of the control approximation NN is 1+1+5 for the systemic risk experiment and 1+2+5 for the price impact and crowd motion experiments. As it is mentioned in the main text, especially for the empirical + FFNN implementation, we encounter a very high dimensional problem where the total number of input dimensions are 2+1000 (1 dimension for time, 1 dimension for state and 1000 dimension for the empirical approximation) for the systemic risk experiment, and 3+2*800 (1 dimension for time, 2 dimension for state and 2*800 dimension for the empirical approximation) for the price impact and crowd motion experiments.

\vskip12pt
\noindent\textbf{Intuitive remarks on NN architecture performance:} 
    \begin{itemize}
        \item If we compare empirical, moments, and histogram implementation, we are expecting the empirical and moments to work better than histogram since for histogram approximation, the information we are using is the number of particles in the specific bins. However, when we use for example empirical approximation, we are using all the particle states as our input which holds more information. Furthermore, the performance of histogram approximation will highly depend on the number of bins and the range of the bins chosen. We believe this is the reason why empirical and moments work better in Figure~\ref{fig:primpact} (left) than the histograms.
        \item For histogram approximations (with FFNN vs CNN) and empirical approximations (with FFNN vs SYM), intuitively we expect the results to improve when we change the neural network architecture from FFNN to convolutional neural network (CNN) if we have histogram approximation and to symmetric neural network (SYM) if we have empirical approximation in complex applications. The reason for this is as follows: for histogram, we would like to keep the spatial dependencies and CNN helps us with this; for empirical, the order of the particle states should not be important and symmetric neural network helps us to implement this since it is invariant with respect to the permutations of the particle states. This is what we see in Figure~\ref{fig:crowdmotion} (left) when we handle a more complex model. 
    \end{itemize}
\clearpage
\noindent{\textbf{Computing resources:}}
The experiments are run on a HPC cluster with the following properties:
\begin{center}{\small
\begin{tabular}{ |l |l|}
\hline
    Model & Dell  \\ \hline
 CPU &  Intel Xeon Gold 6226 2.9 Gh \\  \hline
 Number of CPUs &  $2$ \\  \hline 
 Cores per CPU &    $16$\\  \hline 
 Total cores &  $32$ \\ \hline 
 Memory &   $192$ GB \\  \hline 
 Network &   EDR Infiniband \\  \hline 
\end{tabular}}
\end{center}

\section{Parameters}
\label{app:parameters}

The model parameters used in the numerical experiments can be found below in tables~\ref{table:sysrisk_coef}, \ref{table:primpact_coefs}, and \ref{table:crowmo_coefs}.
\begin{table}[h]
\caption{Parameters in the systemic risk MFC experiments with common noise\\}
\centering
{
\begin{tabular}{ccccccccc}
\toprule
 $T$ & $\Delta t$ & $\mu_0$ & $\rho$ & $a$ & $c$ & $q$ & $\varepsilon$ &$\sigma$\\
\midrule
$1.0$ & $0.01$ & $\mathcal{N}(1,0.1^2)$ & $0.1$ & $1.0$ & $1.0$ & $0.5$ & $10.0$ & 1.0 \\
\bottomrule
\end{tabular}}
\label{table:sysrisk_coef}
\end{table}

\begin{table}[h]
\caption{Parameters in the price impact MFC experiments\\}
\centering
{
\begin{tabular}{ccccccc}
\toprule
 $T$ & $\Delta t$  & $\mu^1_0$ & $\mu_0^2$ & $c_\alpha$ & $c_X$ & $c_g$  \\
\midrule
$1.0$ & $0.01$ & $\mathcal{N}(1,0.3^2)$ & $\mathcal{N}(2,1)$ & $2.0$ & $0.1$ & $0.3$  \\
\toprule
$h_1$ & $h_2$ & $x_0^1$ & $x_0^2$ & $\sigma_1$ & $\sigma_2$\\
\midrule
$1.0$ & $0.8$ &$\mathcal{N}(0,1)$ &$\mathcal{N}(0,1)$ &$0.7$ & $1.0 $\\
\bottomrule
\end{tabular}}
\label{table:primpact_coefs}
\end{table}

\begin{table}[h]
\caption{Parameters in the crowd motion MFC experiments\\}
\centering
{
\begin{tabular}{cccccccc}
\toprule
 $T$ & $\Delta t$  & $\mu^1_0$ & $\mu_0^2$ & $c_0$ & $c_1$ & $c_2$ \\
\midrule
$1.0$ & $0.01$ & $\mathcal{N}(0,0.1^2)$ & $\mathcal{N}(0,0.2^2)$ & $0.1$ & $0.0$ & $1.0$    \\
\toprule
 $c_3$ & $x_{target2}$ & $x_0^1$ & $x_0^2$  & $\sigma_1$ & $\sigma_2$\\
\midrule
$1.0$ & $[2,2]$ &$\mathcal{N}(0,1)$ &$\mathcal{N}(0,1)$&$ 0.7$ & $1.0$\\
\bottomrule
\end{tabular}}
\label{table:crowmo_coefs}
\end{table}

         \clearpage   

\section{Additional Experiment Results}

In this section, we give the results to the additional experiments. For the systemic risk model, we focus on FFNN with histogram approximation and we compare the effect of different learning rates ($10^{-2}$ vs. $10^{-3}$ vs. $10^{-4}$). In Figure~\ref{fig:sysrisk-additional-1}, we can see that when learning rate is equal to $10^{-4}$, the algorithm converges slower as expected but it still converges around the same loss level as with the other learning rates.

In the top plot of Figure~\ref{fig:primpact-additional}, we can see the states vs. state-dependent controls at different time points and at bottom, we can see the states vs.
population-dependent controls (specifically when we used FFNN with empirical approximation for distribution embedding). From the loss plots (Figure~\ref{fig:primpact} (left)), we know that the population-dependent controls perform better. 

In  Figure~\ref{fig:crowdmotion-additional}, we show the positions of the particles at 3 different time steps ($t=0.0, 0.5, 1.0$). The terminal target location is shown at the intersection of orange dashed lines and we can see that the crowd moves towards this target with time.

\begin{figure}[h]
\centering
\begin{subfigure}[t]{0.55\textwidth}
    \centering
    \includegraphics[width=\linewidth]{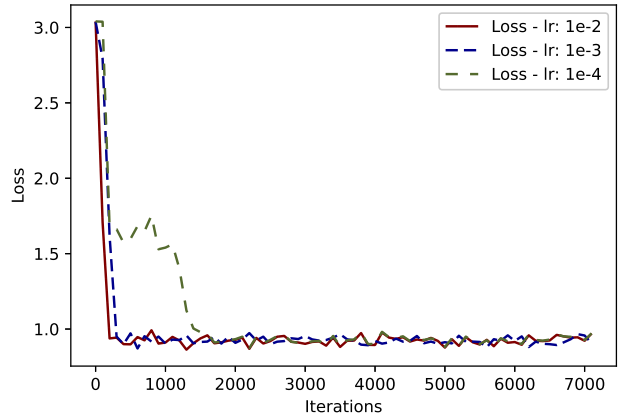}
\end{subfigure}
    \caption{Loss comparison of distribution approximation with histogram with different learning rates ($10^{-2}, 10^{-3}, 10^{-4}$) in the systemic risk experiment.\label{fig:sysrisk-additional-1}}
\end{figure}

\begin{figure}[h]
\centering
\vskip-3.3cm
\includegraphics[width=0.95\linewidth]{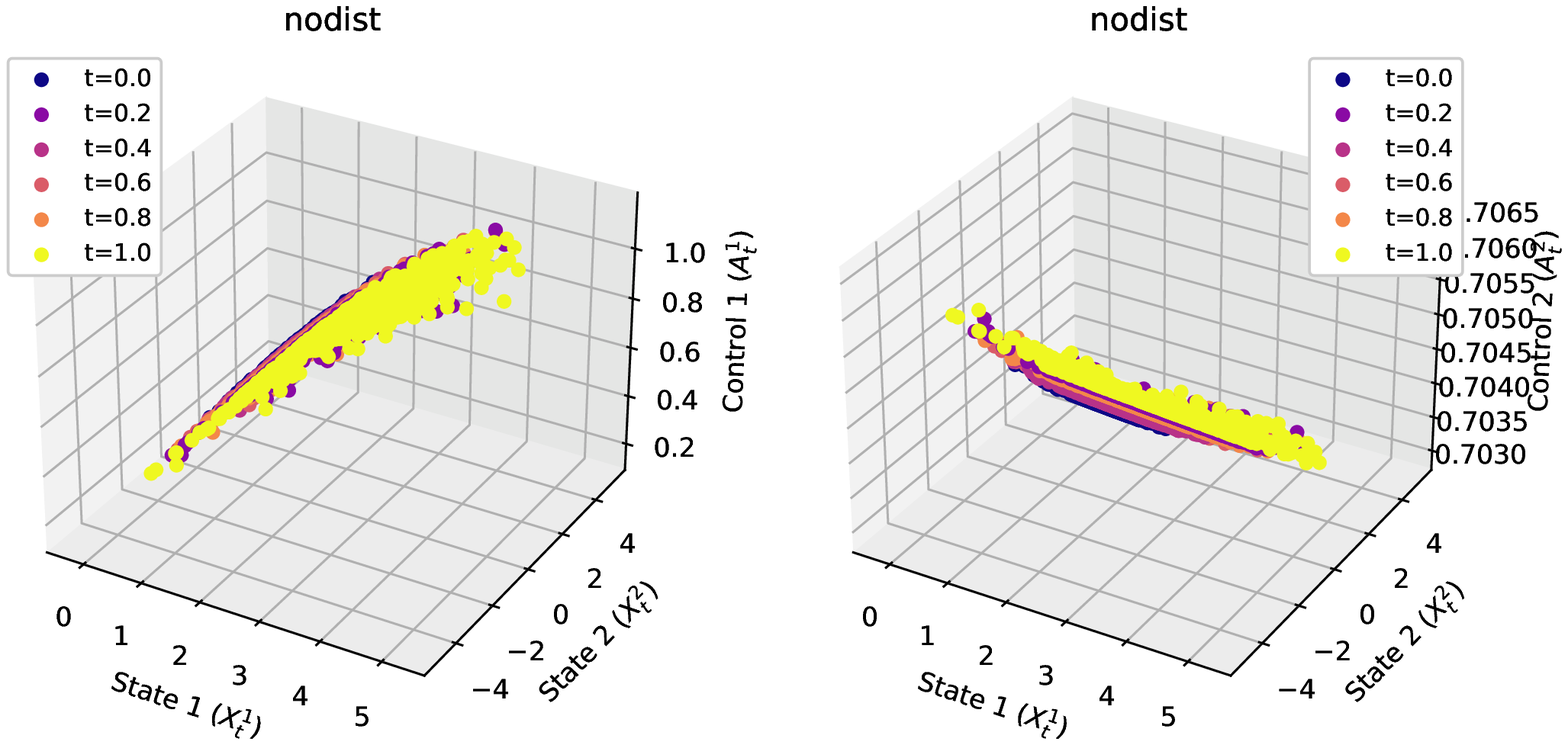}
\vskip-2.6cm
\includegraphics[width=0.95\linewidth]{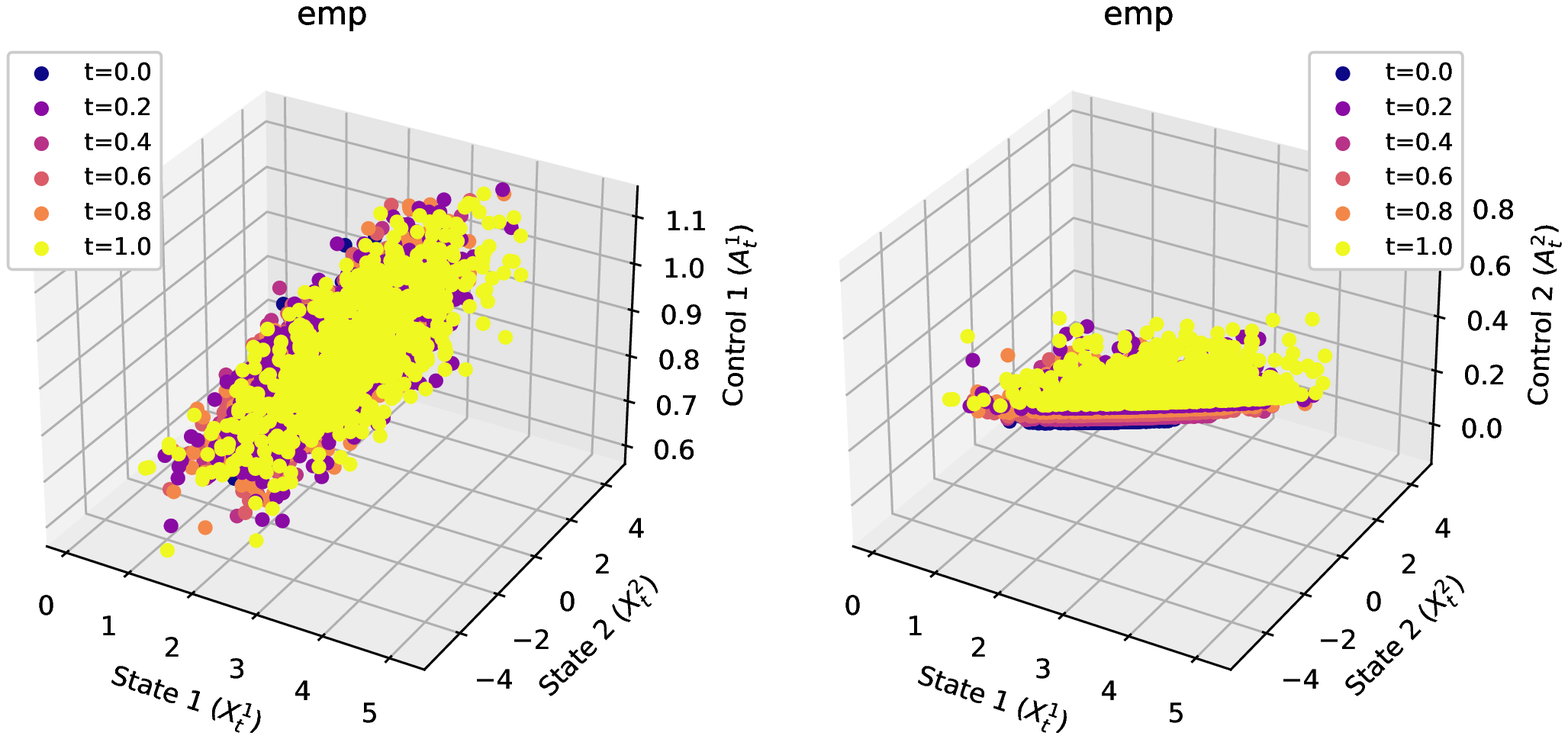}
\vskip-2cm
\caption{Comparison of control vs state plots. \textbf{Top:} where no distribution approximation is used as an input for control (i.e., state-dependent control) vs. \textbf{Bottom:} where empirical distribution approximation is used as an input for control (population-dependent control).}
\label{fig:primpact-additional}
\end{figure}

\begin{figure}[h]
\centering
\includegraphics[width=\linewidth]{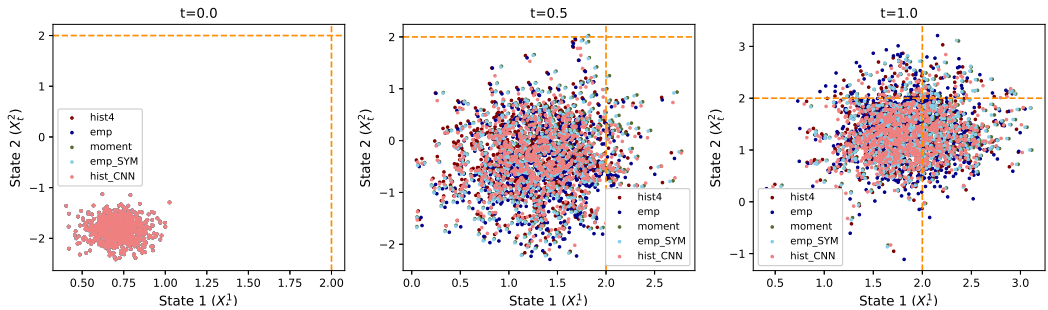}
\caption{Locations over time in the crowd motion example.}
\label{fig:crowdmotion-additional}
\end{figure}

\end{document}